\documentclass[11pt,a4paper]{article}

\usepackage{authblk}
\usepackage[english]{babel}
\usepackage[OT1]{fontenc}
\usepackage{graphicx}
\usepackage{indentfirst}
\usepackage{amsmath,amssymb,amsthm,empheq}
\usepackage{ dsfont }
\usepackage[italian]{varioref}
\usepackage{booktabs}
\usepackage{siunitx}
\usepackage{microtype}
\usepackage{epstopdf}
\usepackage[margin=1in]{geometry} 
\usepackage[hidelinks]{hyperref}
\usepackage{fancyhdr}
\usepackage{version}
\usepackage[utf8]{inputenc}
\usepackage{tikz}
\usetikzlibrary{positioning, shapes.geometric, arrows.meta, calc, backgrounds}
\usepackage{geometry} 
\usepackage{tikz-cd}
\usepackage[mathscr]{eucal}
\usepackage{enumerate}
\usepackage{addfont}
\usepackage{aurical}
\usepackage[titletoc]{appendix}

\usepackage{tocloft}
\usepackage[autostyle,italian=guillemets]{csquotes}
\usepackage[
backend=biber,
style=numeric, 
]{biblatex}
\DeclareNameFormat{author}{%
	\namepartgiveni 
	\adddot\space   
	\namepartfamily 
	\ifthenelse{\value{listcount}<\value{liststop}} 
	{\addcomma\space} 
	{} 
}

\definecolor{darkred}{RGB}{175,0,0}
\newcommand{\K}{\mathbb{K}}
\newcommand{\R}{\mathbb{R}}

\newcommand{\N}{\mathbb{N}}

\newcommand{\Q}{\mathbb{Q}}

\newcommand{\norm}[1]{\|#1\|}
\newcommand{\comp}{\subset\subset}
\newcommand{\cont}{\mathcal{C}}
\renewcommand{\epsilon}{\varepsilon}
\renewcommand{\phi}{\varphi}

\newcommand{\Per}{\operatorname{Per}}

\newcommand{\loc}{\rm{loc}}

\renewcommand{\K}{\text{\Fontskrivan K}} 
\newcommand{\J}{\text{\Fontauri J}}
\newcommand{\F}{\text{\Fontauri F}} 
\newcommand{\G}{\text{\Fontauri G}}
\renewcommand{\H}{\text{\Fontskrivan H}}
\newcommand{\E}{\text{\Fontauri E}} 
\renewcommand{\P}{\text{\Fontauri P}}

\theoremstyle{plain}
\newtheorem{theorem}{Theorem}[section]
\newtheorem{lemma}[theorem]{Lemma}
\newtheorem{prop}[theorem]{Proposition}
\newtheorem*{prop*}{Proposition}
\newtheorem{cor}[theorem]{Corollary}
\theoremstyle{definition}
\newtheorem{definition}[theorem]{Definition}
\newtheorem*{definition*}{Definition}
\theoremstyle{remark}
\newtheorem{rem}[theorem]{Remark}
\newtheorem{ex}[theorem]{Example}

\newtheorem*{recall*}{Recall}
\numberwithin{equation}{section}

\title{$\Gamma$-convergence of the non-local Massari functional \\
	and 
	applications to inhomogeneous Allen-Cahn equations}
\date{}

\author{\href{https://research-repository.uwa.edu.au/en/persons/serena-dipierro}{Serena Dipierro}}

\author{\href{https://research-repository.uwa.edu.au/en/persons/enrico-valdinoci}{Enrico Valdinoci}}

\author{{Riccardo Villa}}

\affil{ {\footnotesize Department of Mathematics and Statistics,}\\
	{\footnotesize The University of Western Australia,} 
	{\footnotesize 35 Stirling Highway,
		Perth, WA 6009, Australia}\\
	{\footnotesize\tt serena.dipierro@uwa.edu.au},
	{\footnotesize\tt enrico.valdinoci@uwa.edu.au},
	{\footnotesize\tt riccardo.villa@research.uwa.edu.au}}

\begin{document}
	
	\pagenumbering{gobble}
	\maketitle
	
	\begin{abstract}
		We present several asymptotic results concerning the non-local Massari Problem for sets with prescribed mean curvature. In particular, we show that the fractional Massari functional~$\Gamma$-converges to the classical one, and this convergence
		preserves minimizers in the~$L^1_{\loc}$-topology.
		
		This returns useful information about the asymptotic behavior of the solutions of the inhomogeneous Allen-Cahn equation in the forced and the mass-prescribed settings. In this context, a new geometric object, which we refer to as ``non-local hybrid mean curvature'', naturally appears.
	\end{abstract}
	\setcounter{tocdepth}{1}
	\tableofcontents
	
	\pagenumbering{arabic}
	\setcounter{page}{1}
	
	\section{Introduction and main results} \label{sec::intro}
	In this paper we prove some asymptotic results involving the non-local Massari Problem for sets with prescribed mean curvature and obtain some interesting applications to the space-fractional Allen-Cahn equation. 
	
	The Massari Problem has been firstly formulated in the classical case by Umberto Massari in the mid 70s and provides a variational definition of mean curvature when the latter is a summable, but not necessarily continuous function (see e.g.~\cite{MR0355766, MR1250498}). Recently, the fractional counterpart has also been investigated in~\cite{MR5015001}.
	\medskip
	
	The mathematical details are as follows.	
	For the rest of our discussion, we denote by~$\Omega$ a bounded, Lipschitz, open set of~$\R^n$. Besides, whenever needed, all subsets of~$\R^n$ will be implicitly assumed to be Lebesgue-measurable.

	We recall that, given a function~$H\in L^1(\Omega)$, the Massari functional~$\mathscr{J}_{\Omega}^H$ introduced in~\cite{MR0355766} is defined for every set~$E$ as
	\begin{equation} \label{eq::classic_massari}
		\mathscr{J}_{\Omega}^H(E) := \Per(E,\Omega) + \frac{1}{\omega_{n-1}}\int_{\Omega\cap E} H\, dx ,
	\end{equation}
	where~$\Per$ denotes the variational perimeter, and~$\omega_{n-1}$ is the volume of the unitary ball in~$\R^{n-1}$.
	
	In the wake of~\cite{caffarelli_roquejoffre_savin_nonlocal}, for~$s\in\left(0,\frac12\right)$, we also denote the non-local interaction between a couple of disjoint sets~$A$ and~$B$ by 
	\begin{equation*}
		\mathcal{L}(A,B) := \int_A \int_B \frac{dxdy}{|x-y|^{n+2s}} ,
	\end{equation*}
	and the~$s$-Perimeter of a set~$E$ by
	\begin{equation*}
		\Per_s(E,\Omega) :=  \mathcal{L}(E\cap\Omega,E^c\cap\Omega) + \mathcal{L}(E\cap\Omega,E^c\cap\Omega^c) + \mathcal{L}(E\cap\Omega^c,E^c\cap\Omega) .
	\end{equation*}
	
	Thus, similarly to~\eqref{eq::classic_massari}, the non-local Massari functional~$\mathscr{J}_{s,\Omega}^H$ is defined in~\cite{MR5015001} for every set~$E$ as
	\begin{equation} \label{eq::fracional_massari}
		\mathscr{J}_{s,\Omega}^H(E) := \Per_s(E,\Omega) + \frac{1}{1-2s}\int_{\Omega\cap E} H\, dx .
	\end{equation}	
	
	In our study, we focus on:
	\begin{enumerate}[(i)]
		\item proving a general~$\Gamma$-convergence result for a suitably rescaled non-local Massari functional, when the fractional parameter~$s$ approaches~$(1/2)^-$;
		\item showing that the limit of a convergent sequence~$\{E_j\}_j$ of minima for~$\mathscr{J}_{s_j,\Omega}^{H_j}$ is a minimum for the limit functional;
		\item investigating the asymptotic behavior of the solutions of an inhomogeneous Allen-Cahn equation;
		\item discussing some properties of blow-down solutions of a prescribed-mass Allen-Cahn equation. In particular, in this context, we deal with a new geometric object which we refer to as a ``non-local hybrid mean curvature''.
	\end{enumerate}
		The outline of the corresponding results goes as follows.
		\smallskip
		
			\begin{figure}[!h]
				\centering
				\begin{tikzpicture}[
					node distance = 1.2cm, 
					box/.style = {
						rectangle, 
						draw, 
						text width = 6.5cm, 
						align = center, 
						inner sep = 5pt, 
						font = \small,
						minimum height = 1cm
					},
					smallbox/.style = {
						rectangle, 
						draw, 
						text width = 2cm, 
						align = center, 
						font = \small,
						inner sep = 2pt
					},
					mediumbox/.style = {
						rectangle, 
						draw, 
						text width = 5cm, 
						align = center, 
						inner sep = 5pt, 
						font = \small,
						minimum height = 1cm
					},
					arrow/.style = {
						-{Stealth[scale=1.2]}, 
						line width = 0.6pt,
						shorten > = 3pt, 
						shorten < = 3pt  
					}
					]
					
					\node (th_gamma_conv_J) [mediumbox] {$\Gamma$-convergence for the nonlocal Massari's Problem as $s \to \frac{1}{2}^-$ \\ \textbf{(Theorem~\ref{th::gamma_conv_J})}};
					\node (th_min_conv_J) [mediumbox, right = 1.5cm of th_gamma_conv_J] {Stability for minimizers of the nonlocal Massari's functional \\ \textbf{(Theorem~\ref{th::min_J_conv})}};
					
					\draw [arrow] (th_gamma_conv_J) -- (th_min_conv_J);
					
					\draw [dashed, line width=0.5pt] ($(th_gamma_conv_J.south west) + (-0.5,-0.6)$) -- ($(th_min_conv_J.south east) + (0.5,-0.6)$);
					
					\node (inhom) [box, below = 1.8cm of $(th_gamma_conv_J.south)!0.5!(th_min_conv_J.south)$] {Asymptotic behavior for an inhomogeneous Allen-Cahn Equation};
					
					\node (r2) [smallbox, right = 0.8cm of inhom] {$s = \frac{1}{2}$};
					\node (r1) [smallbox, above = 0.2cm of r2] {$s \in \left(0, \frac{1}{2}\right)$};
					\node (r3) [smallbox, below = 0.2cm of r2] {$s \in \left(\frac{1}{2}, 1\right)$};
					
					\draw [arrow] (inhom.east) -- (r1.west);
					\draw [arrow] (inhom.east) -- (r2.west);
					\draw [arrow] (inhom.east) -- (r3.west);
					
					\node (th_gamma_conv_AC) [mediumbox, below left = 1.5cm and -2cm of inhom] {$\Gamma$-convergence for a perturbed fractional Allen-Cahn energy \\ \textbf{(Theorem~\ref{th::gamma_conv})}};
					\node (th_min_conv_AC) [mediumbox, below right =1.5cm and -2cm of inhom] {Stability for minimizers of a perturbed fractional Allen-Cahn energy \\ \textbf{(Theorem~\ref{th::min_conv})}};
					
					\draw [arrow] (inhom.south) -- (th_gamma_conv_AC.north);
					\draw [arrow] (inhom.south) -- (th_min_conv_AC.north);
					
					\coordinate (merge) at ($(th_gamma_conv_AC.south)!0.5!(th_min_conv_AC.south) + (0,-1cm)$);
					
					\node (blow) [box, below = 1.2cm of merge] {Blow-down solutions of a mass-prescribed Allen-Cahn Equation};
					
					\draw [arrow] (th_gamma_conv_AC.south) -- (merge) -- (blow.north);
					\draw [arrow] (th_min_conv_AC.south) -- (merge) -- (blow.north);
					
					\node (th_lagr) [box, below = of blow] {Energy bounds for solutions of a mass-prescribed Allen-Cahn Equation \\ and \\ convergence of the Lagrange multipliers \\ \textbf{(Theorems~\ref{th::energy_uniform_estimate} and~\ref{th::multiplier_bound})}};
					
					\node (th_perturbed_energy) [box, below = of th_lagr] {Energy limits subject to mass constraints \\ \textbf{(Theorem~\ref{th::energy_continuity})}};
					
					\node (th_conv_crit_pts) [box, below = of th_perturbed_energy] {Stability for critical points of the perturbed energy $\mathcal{G}_\epsilon$ \\ \textbf{(Theorem~\ref{th::conv_crit_points})}};
					
					\draw [arrow] (blow) -- (th_lagr);
					\draw [arrow] (th_lagr) -- (th_perturbed_energy);
					\draw [arrow] (th_perturbed_energy) -- (th_conv_crit_pts);
					
				\end{tikzpicture}
		\end{figure}
	
	\subsection{$\Gamma$-convergence of the Massari's problem}
	The first~$\Gamma$-convergence result deals with the case~$s\to(1/2)^-$ and is contained in Theorem~\ref{th::gamma_conv_J}. This result is perhaps not surprising
	in itself, and its proof is an offspring of the methods in~\cite{ambrosio_dephilippis_martinazzi},
	but it is instrumental to establish the limit behavior of the functional under consideration
	and it paves the way to the analysis of the minimizers of the non-local Massari Problem.
	
	This analysis is started by Theorem~\ref{th::min_J_conv}, in which we show that the limit
	of minimizers is a minimizer of the limit functional, with uniform energy bounds. The methods of~\cite{ambrosio_dephilippis_martinazzi} are helpful in the proof of this result too, but
	here one also needs some tailor-made, and somewhat delicate, estimates for 
	the new mean curvature contributions.
	
	\subsection{A perturbed inhomogeneous non-local Allen-Cahn equation}
	Then, we turn our attention to a non-local Allen-Cahn equation
	with a ``highly oscillatory'' forcing term
	and we detect its homogenized behavior in the limit.
	This setting can be seen as the phase transition counterpart of the Massari Problem for a small singular perturbation parameter~$\varepsilon$.
	We establish a~$\Gamma$-convergence result	as~$\varepsilon\to0^+$ in Theorem~\ref{th::gamma_conv}
	and the corresponding stability of minimizers in Theorem~\ref{th::min_conv}.
	
	As a matter of fact, Theorem~\ref{th::gamma_conv} is again perhaps not surprising in itself
	and its proof heavily relies on~\cite{savin_valdinoci_gamma_conv}, but it serves
	to establish the appropriate limit functional. The proof of Theorem~\ref{th::min_conv} is instead much more delicate
	and requires specific pointwise and energy estimates of uniform type. 
	For this, we will treat the cases $s\in\left(0,\frac{1}{2}\right)$,~$s=\frac{1}{2}$, and~$s\in\left(\frac{1}{2},1\right)$ separately.
	Also, the regime~$s\in\left[\frac{1}{2},\,1\right)$ will require additional effort since non-local contributions do localize in the limit.
	
	\subsection{A prescribed-mass Allen-Cahn problem}
	We also consider problems with mass constraints.
	The corresponding~$\Gamma$-convergence result for singularly perturbed non-local phase transition model is provided in Theorem~\ref{th::gamma_conv_mass_prescribed}.
	The delicate part of this result consists in the analysis of the recovery sequence
	when~$s\in\left[\frac{1}{2},\,1\right)$, since in this case the existing literature falls short.
	Specifically, the recovery sequence constructed in~\cite{savin_valdinoci_gamma_conv} accounted for the Dirichlet external condition and instead, in our case, 
	suitable adjustments have to be performed to satisfy the mass condition: the impact
	of these modifications in the ``$\limsup$-inequality'' needs to be carefully controlled
	and this makes the proof somehow involved.
	
	As a byproduct of Theorem~\ref{th::gamma_conv_mass_prescribed}, we obtain a
	uniform energy bound for the non-local phase transition energy under the mass constraint,
	which is stated explicitly in Theorem~\ref{th::energy_uniform_estimate}.
	The miminizers of this problem fulfill a non-local external condition of Neumann type,
	as detailed in Theorem~\ref{th::u_def_outside_omega}.
	
	Interestingly, critical points of the mass constrained
	problem satisfy a non-local phase transition equation with a Lagrange multiplier.
	This additional term also depends on the singular perturbation parameter~$\varepsilon$
	and it is therefore natural to investigate its asymptotic as~$\varepsilon\to0^+$. This problem is
	addressed in Theorem~\ref{th::multiplier_bound}, whose proof requires
	appropriate uniform estimates for this scope.
	
	The convergence of these Lagrange multipliers also produces
	a limit functional, and one would like to relate the corresponding energies in the limit.
	This question is dealt with in Theorem~\ref{th::energy_continuity},
	where it is shown that minimizers for the~$\varepsilon$-perturbed functional
	approach globally (and not only locally, thanks to a non-local Neumann condition)
	a limit functional, and that the minimal energy converges to the energy of this limit function.
	
	A nontrivial aspect of the asymptotics in the problems with mass prescription
	is that, while the large interactions provided by the double-well potential force
	the minimizers to attain ``pure phases'' in the domain of reference,
	outside this domain, minimizers can still oscillate and do not reduce
	to piecewise constant functions. 
	
	This is an interesting feature, showcased in
	Theorem~\ref{th::conv_crit_points}, giving rise to a new geometric object,
	which we name ``non-local hybrid mean curvature''.
	In a nutshell, while classical problems of phase transitions with a mass constraint
	are related to constant mean curvature limit configurations,
	in our setting the limit consists of a couple, given by
	a set inside the domain and of a function outside. The non-local hybrid mean curvature
	corresponds to a singular kernel integral involving this couple;
	limit configurations are related here to the
	constancy of the non-local hybrid mean curvature.
	
	The proof of Theorem~\ref{th::conv_crit_points} relies on a subtle refinement of the results in~\cite{MR3900821}, which are not directly applicable in our framework due to a lack of global regularity.
	We believe that this improvement is of independent interest and is formalized in Theorem~\ref{th::millot_sire_wang}.
	\medskip
	
	Let us now dive into the technical details of our results.
	
	\section{Main results} \label{sec::main_results}
	
	\subsection{$\Gamma$-convergence for the Massari Problem}
	The notion of~$\Gamma$-convergence has been introduced by De Giorgi as an alternative to the convergence in the sense of Kuratowski ($K$-convergence). This has been an intense field of research in the recent years, finding many interesting applications in the fields of partial differential equations and mathematical physics (see for instance~\cite{modica_phase_transition1, modica_phase_transition2, homogeneization, MR2582099, savin_valdinoci_gamma_conv}).
	\medskip
	
	We recall the following classical definition of~$\Gamma$-convergence (see~\cite[Proposition~8.1]{dal_maso}).
	\begin{definition} \label{def::gamma_conv}
		Let~$V$ be a topological space that satisfies the first axiom of countability. Consider a sequence of functions~$\{F_h:V\to\overline{\R}\}_h$, and a function~$F:V\to\overline{\R}$.
		
		We say that~$F_h$ $\Gamma$-converges to~$F$, and we write~$F_h\xrightarrow{\Gamma}F$, if the following properties hold:
		\begin{enumerate}[(i)]
			\item\label{enum::gamma_conv_inf} for every~$x\in V$ and for every sequence~$\{x_h\}_h$ in~$V$ converging to~$x$, it holds that
			$$ \liminf_{h\to+\infty} F_h(x_h) \geq F(x)\ ;$$
			\item\label{enum::gamma_conv_sup} for every~$x\in V$ there exists a sequence~$\{x_h\}_h\subset V$ such that~$x_h\to x$ and
			$$ \limsup_{h\to+\infty} F_h(x_h) \leq F(x)\ .$$
		\end{enumerate}
	\end{definition}
	
	We also recall from~\cite{MR5015001} that a solution of the non-local Massari Problem is a set~$E$ satisfying
	\begin{equation}\label{eq::Sdef1}
		\mathscr{J}_{s,\Omega}^H (E) \leq \mathscr{J}_{s,\Omega}^H (F) \text{ for every~$F$ such that }F\setminus\Omega=E\setminus\Omega.
	\end{equation}
	
	Moreover, the following definition of non-local almost minimizers has been given in~\cite[Definition~1.1]{MR5015001}:
	
	\begin{definition}\label{def11}(Non-local almost minimizers). Given a parameter~$\Lambda\ge 0$, we say that a set~$E$
		is a~$\Lambda$-minimal set (or is an almost minimal set with respect to~$\Lambda$) for the~$s$-Perimeter in~$\Omega$ if
		$$  \Per_s(E,\Omega)\le \Per_s(F,\Omega) + \Lambda|E\Delta F|,$$
		for every set~$F$ such that~$E\Delta F\subset\Omega$.
	\end{definition}
	
	We have that, when~$H\in L^\infty(\R^n)$, 
	a solution~$E$ of~\eqref{eq::Sdef1} is a~$\Lambda$-minimal set, with~$\Lambda:=\norm{H}_{L^\infty(\Omega)}$. Hence, $\partial E$ is a~$\cont^{1,\alpha}$-hypersurface up to a set of Hausdorff dimension~$n-3$, for every~$\alpha\in(0,s)$, and has variational non-local mean curvature~$H$ in~$\Omega$ (see~\cite[Section~7]{MR5015001} for the details).
	\smallskip
	
	Important~$\Gamma$-convergence results in the setting of minimal surfaces have been developed in~\cite{ambrosio_dephilippis_martinazzi}, where it is also proved that a suitable rescaling of the~$s$-Perimeter~$\Gamma$-converges to the classical one, as~$s\to(1/2)^-$, that a convergent sequence of~$s$-minimal surfaces has uniformly bounded energy,
	and that its limit is a minimal surface in the classical sense
	(see~\cite{MR3586796, MR1942130, MR2782803}
	for related results).
	
	We show here that analogous results hold also for the Massari Problem:
	
	\begin{theorem}[$\Gamma$-convergence of the non-local Massari Problem]
		\label{th::gamma_conv_J}
		Let~$\{s_j\}_j\subset \left(0,\frac12\right)$ be a sequence such that~$s_j\to(1/2)^-$. Let~$\{H_j\}_j\subset L^\infty(\Omega)$ be
		such that~$\sup_j\norm{H_j}_{L^\infty(\Omega)}<+\infty$, and~$H_j\to H$ in~$L^1(\Omega)$, for some function~$H$.
		
		Then, for every set~$E\subseteq\R^n$, we have that:
		\begin{enumerate}[(i)]
			\item for every sequence of sets~$\{E_j\}_j$ such that~$E_j\to E$ in~$L^1_{\loc}(\R^n)$, it holds that
			$$ \omega_{n-1}\mathscr{J}_{\Omega}^H(E) \leq \liminf_{j\to+\infty}(1-2s_j) \mathscr{J}_{s_j,\Omega}^{H_j}(E_j) ,$$
			\item there exists a sequence of sets~$\{E_j\}_j$ such that~$E_j\to E$ in~$L^1_{\loc}(\R^n)$ and
			$$ \omega_{n-1}\mathscr{J}_{\Omega}^H(E) \geq \limsup_{j\to+\infty}(1-2s_j) \mathscr{J}_{s_j,\Omega}^{H_j}(E_j) .$$
		\end{enumerate}
	\end{theorem}
	
	\begin{theorem}[Energy bounds and limit stability of minimizers for the non-local Massari Problem]\label{th::min_J_conv}
		Let~$\{s_j\}_j\subset\left(0,\frac12\right)$ be a sequence such that~$s_j\to(1/2)^-$, and~$\{H_j\}_j\subset L^\infty(\Omega)$ be such that~$\sup_j\norm{H_j}_{L^\infty(\Omega)}<+\infty$, and~$H_j\to H$ in~$L^1(\Omega)$.
		
		Let~$\{E_j\}_j$ be a sequence of sets such that~$E_j$ is a local minimizer for~$\mathscr{J}_{s_j,\Omega}^{H_j}$, and~$E_j\to E$ in~$L^1_{\loc}(\R^n)$, for some set~$E\subseteq\R^n$. 
		
		Then,
		\begin{equation} \label{eq::J_uniform_bound}
			\limsup_{j\to+\infty}\, (1-2s_j)\mathscr{J}_{s_j,K}^{H_j}(E_j)<+\infty ,\quad \text{for every } K\comp\Omega .
		\end{equation}
		
		Moreover, $E$ is a local minimizer for the classical Massari functional, and 
		\begin{equation} \label{eq::J_cont_conv_min}
			\lim_{j\to+\infty} (1-2s_j)\mathscr{J}_{s_j,K}^{H_j}(E_j) = \omega_{n-1}\mathscr{J}_{K}^{H}(E)\ ,
		\end{equation}
		for every~$K\comp\Omega$ such that~$\Per(E,\partial K)=0$.
	\end{theorem}
	
	We point out that the factor~$1-2s$ in Theorems~\ref{th::gamma_conv_J} and~\ref{th::min_J_conv} is the appropriate scaling factor to obtain interesting~$\Gamma$-limit properties.
	
	The proofs of Theorems~\ref{th::gamma_conv_J} and~\ref{th::min_J_conv} will be presented in Sections~\ref{sec::proof_gamma_conv_J} and~\ref{sec::proof_min_J_conv} respectively. 
	
	\subsection{An alternative formulation for the Massari Problem}
	
	Since we are interested in the~$\Gamma$-convergence for inhomogeneous Allen-Cahn equations, we recall that the limit of a convergent sequence of minimizers for the Allen-Cahn energy functional takes the form~$\chi_E-\chi_{E^c}$, for some set~$E$ (see~\cite[Theorem~1.3]{savin_valdinoci_gamma_conv}). As an analogous result holds true in our setting, in order to incorporate the extra term due to the term~$-\chi_{E^c}$ in the~$\Gamma$-limit, we now propose an alternative formulation of the Massari Problem to the one presented in~\eqref{eq::fracional_massari}.
	
	With this aim, notice that, by definition of~$s$-Perimeter, for every set~$E$, we have
	\begin{equation*}
		\Per_s(E,\Omega) = \Per_s(E^c,\Omega),
	\end{equation*}
	hence
	\begin{equation*}
		\mathscr{J}_{s,\Omega}^H (E) - \frac{1}{1-2s}\int_\Omega Hdx = \Per_s(E^c,\Omega)-\frac{1}{1-2s}\int_{\Omega\cap E^c} Hdx = \mathscr{J}_{s,\Omega}^{-H} (E^c).
	\end{equation*} 
	This justifies the intuition that a set~$E$ has non-local mean-curvature~$H$ if and only if~$E^c$ has non-local mean-curvature~$-H$.
	
	Moreover, we deduce that~$E$ is a solution of the non-local Massari Problem according to~\eqref{eq::Sdef1} if and only if~$E$ is a minimizer of 
	\begin{equation}\label{gufiwyfew7965tyhieshrdef}
		\P_{s,\Omega}^H(F) := \Per_s(F,\Omega) + \frac{1}{2(1-2s)}\int_\Omega (\chi_F-\chi_{F^c})H dx ,
	\end{equation}
	among all sets~$F$ such that~$F\setminus\Omega=E\setminus\Omega$.
	
	Similarly, a set~$E$ is a solution of the classical Massari Problem if and only if it is a minimizer of
	\begin{equation}\label{gufiwyfew7965tyhieshrdef2}
		\P_{\Omega}^H(F) := \Per(F,\Omega) + \frac{1}{2\omega_{n-1}}\int_\Omega (\chi_F-\chi_{F^c})H dx ,
	\end{equation}
	among all the sets~$F$ such that~$F\setminus\Omega=E\setminus\Omega$.
	
	Recalling what discussed in the Introduction of~\cite{MR0355766} and~\cite[Definition~1.14]{MR5015001}, we have the following:
	\begin{definition} \label{def::mean_curv}
		Let~$H$ be a function in~$L^1(\Omega)$. We say that a set~$E\subseteq\R^n$ has (variational) mean curvature~$H$ if~$E$ is a minimizer of~$\P_{\Omega}^H$ among all the set which coincides with~$E$ outside~$\Omega$.
		
		Similarly, we say that a set~$E$ has non-local (variational) mean curvature~$H$ if~$E$ is a minimizer of~$\P_{s,\Omega}^H$ among all the sets which coincide with~$E$ outside~$\Omega$. 
	\end{definition}
	
	\subsection{A perturbed inhomogeneous Allen-Cahn equation}
	$\Gamma$-convergence represents a key tool in handling singularly perturbed energies. One of the main examples of the latter comes from the gradient theory of phase transitions (see~\cite{modica_mortola, modica_phase_transition1, MR870014}), where the total energy is given by a ``double-well potential''~$W$ with two isolated minima plus a gradient term that prevents the formation of unnecessary interfaces between the two phases. Such energy is related to the so-called Allen-Cahn equation and takes the form
	$$ \int_\Omega \left[\frac{\epsilon^2}{2}|\nabla u|^2 + W(u)\right]dx,\quad \text{with }\epsilon\to0^+.$$
	
	In this framework, the assumptions that we take on the potential~$W$ are the following: we suppose that~$W:\R\to[0,+\infty)$ satisfies
	\begin{equation}\label{eq::assW}
		\begin{split}
			&W\in C^2(\R), \quad W(\pm1)=0,\quad W>0 \;{\mbox{ in }} \R\setminus\{\pm1\},\\
			& W'(\pm1)=0\quad {\mbox{ and }}\quad W''(\pm1)>0.
		\end{split}
	\end{equation}
	For the sake of concreteness, one can keep in mind the model case~$W(t):=\frac{1}{4}(1-t^2)^2$.
	
	A non-local analog of the Allen-Cahn energy functional, where the gradient term is replaced by a fractional Gagliardo semi-norm, has been addressed in the recent literature (see~\cite{savin_valdinoci_gamma_conv,savin_valdinoci_density_estimates}). 
	
	In this paper, we generalize this non-local setting taking into account a macroscopic forcing term~$H$ and proving a~$\Gamma$-convergence result for the new perturbed energy. Then, we show that~$H$ is suitably related to the variational mean curvature of the minimizers of the~$\Gamma$-limit energy functional.
	
	For our purposes, we adopt some notation introduced in~\cite{savin_valdinoci_gamma_conv, savin_valdinoci_density_estimates}. Given~$M\in(1,+\infty)$, we define
	\begin{equation} \label{eq::def_X_M}
		X_M := \big\{u\in L^1_{\rm{loc}}(\R^n) \text{ s.t. } \norm{u}_{L^\infty(\Omega) } \leq M \big\} ,
	\end{equation}
	which is a Fr\'echet space when endowed with the family of semi-norms~$\{\norm{\cdot}_{L^1(B_r)}\}_{r\in\Q}$.
	
	For every function~$u\in X_M$ and for every couple of sets~$A$, $B\subseteq\R^n$, let
	\begin{equation*}
		u(A,B) := \int_A\int_B \frac{|u(x)-u(y)|^2}{|x-y|^{n+2s}} dxdy .
	\end{equation*}
	Then, we define the functional
	\begin{equation}\label{eq::def_K}
		\begin{split}
			\K(u,\Omega) 
			:=\, &\frac{1}{2}\int_\Omega\int_\Omega \frac{|u(x)-u(y)|^2}{|x-y|^{n+2s}}dxdy + \int_\Omega\int_{\Omega^c} \frac{|u(x)-u(y)|^2}{|x-y|^{n+2s}}dxdy \\
			=\, &\frac{1}{2}u(\Omega,\Omega) + u(\Omega,\Omega^c) ,
		\end{split}
	\end{equation}
	which accounts for the local contribution in~$\Omega$ of the~$H^s$-Gagliardo semi-norm of~$u$.
	
	With this notation, the energy functional~$\J_\epsilon:X_M\to\R\cup\{+\infty\}$ associated with the non-local Allen-Cahn equation is 
	\begin{equation*}
		\J_\epsilon(u,\Omega) := \epsilon^{2s}\K(u,\Omega) + \int_\Omega W(u)dx .
	\end{equation*}
	
	In order to obtain interesting~$\Gamma$-limits for the functional~$\J_\epsilon$ and prevent the energy to diverge when~$\epsilon\to0^+$, we introduce the rescaled energy functional~$\F_\epsilon:X_M\to\R\cup\{+\infty\}$, depending on whether~$s\in(0,1/2)$, $s=1/2$, or~$s\in(1/2,1)$, as
	\begin{equation} \label{eq::F_epsilon}
		\F_\epsilon(u,\Omega) :=
		\begin{cases}
			\epsilon^{-2s} \J_\epsilon(u,\Omega) \quad&\text{if }s\in\left(0,\frac{1}{2}\right), \\
			|\epsilon\log\epsilon|^{-1}\J_\epsilon(u,\Omega) \quad&\text{if }s=\frac{1}{2}, \\
			\epsilon^{-1}\J_\epsilon(u,\Omega) \quad&\text{if }s\in\left(\frac{1}{2},1\right). 
		\end{cases}
	\end{equation}
	
	In~\cite[Theorem~1.2]{savin_valdinoci_gamma_conv}, the authors have shown that	$\F_\epsilon$ $\Gamma$-converges to the functional~$\F:X_M\to\R\cup\{+\infty\}$ defined as
	\begin{equation} \label{eq::F}
		\F(u,\Omega) :=
		\begin{cases}
			\K(u,\Omega) \quad &\text{if $u|_\Omega=\chi_E-\chi_{E^c}$, for some set }E\subset\Omega, \\
			&\quad\text{and $s\in\left(0,\frac{1}{2}\right)$,} \\
			c_\star\Per(E,\Omega) \quad &\text{if $u|_\Omega=\chi_E-\chi_{E^c}$, for some set }E\subset\Omega, \\
			&\quad\text{and $s\in\left[\frac{1}{2},1\right)$,} \\
			+\infty \quad &\text{otherwise,}
		\end{cases}
	\end{equation}
	where~$c_\star=c_\star(n,s,W)$ is a suitable positive constant related to a particular~$1$-dimensional minimal profile (see~\cite[Theorem~4.2 and  Formula (4.35)]{savin_valdinoci_gamma_conv}).
	
	\begin{rem} \label{rem::X_M}
		We point out that the results presented in~\cite{savin_valdinoci_gamma_conv} are stated and proved for functionals~$\F_\epsilon$ and~$\F$ defined on the space 
		$$X := \{u\in L^1_{\loc}(\R^n) \mbox{ s.t. }\norm{u}_{L^\infty(\R^n)}\leq 1\}.$$
		However, by inspection of the proofs contained there, we see that the same results can be proved with identical arguments when considering the space~$X_M$ defined in~\eqref{eq::def_X_M}, for any~$M\geq1$.
		
		In fact, the upper bound imposed in the definition of~$X$ plays no other role than being uniform (in~$\epsilon$). 
		We refer to	Appendix~\ref{append::X_M} for a more detailed and self-contained discussion of this adaptation.
	\end{rem}
	\smallskip
	
	Now, let~$g$ be a function in~$L^\infty(\R^n)$ and consider a ``forced'' non-local phase-transition model with a perturbed fractional Allen-Cahn energy functional given by
	\begin{equation*}
		\epsilon^{2s}\K(u,\Omega) + \int_\Omega \left(W(u)+g_\epsilon u\right) dx,
	\end{equation*}
	where~$g_\epsilon(x):=g(x/\epsilon)$.
	
	Then, we define the sequence~$\{H_\epsilon\}_\epsilon\subset L^\infty(\R^n)$ as
	\begin{equation*}
		H_\epsilon(x) :=
		\begin{cases}
			2\epsilon^{-2s}g_\epsilon(x) \quad&\text{if }s\in\left(0,\frac{1}{2}\right),\\
			2|\epsilon\log(\epsilon)|^{-1}g_\epsilon(x) \quad&\text{if }s=\frac{1}{2},\\
			2\epsilon^{-1}g_\epsilon(x) \quad&\text{if }s\in\left(\frac{1}{2},1\right).
		\end{cases}
	\end{equation*}
	Furthermore, assume that~
	\begin{equation} \label{eq::H_conv_assumption}
		\lim_{\epsilon\to0} \norm{H_\epsilon-H}_{L^1(\Omega)} = 0,
	\end{equation}
	for some macroscopic forcing term~$H\in L^\infty(\R^n)$. 
	
	We consider the forcing term energy contributions~$\H_\epsilon$, $\H:X_M\to\R$ given by
	\begin{eqnarray*}
		\H_\epsilon(u,\Omega) &:=& c_{n,s} \int_\Omega H_\epsilon u\, dx ,\\
		{\mbox{and }}\quad 
		\H(u,\Omega)& :=& c_{n,s}\int_\Omega H\,u\, dx ,
	\end{eqnarray*}
	where
	$$ c_{n,s}:=\begin{cases}
		\displaystyle\frac1{2(1-2s)} & {\mbox{ if }} s\in\left(0,\frac12\right),\\
		\displaystyle\frac1{2\omega_{n-1}} & {\mbox{ if }} s\in\left[\frac12,1\right). 
	\end{cases}$$
	
	We also define the total perturbed energies
	$$\E_\epsilon :=\F_\epsilon + \H_\epsilon\quad{\mbox{
			and}}\quad\E := \F + \H.$$
	
	With this notation, we prove the following results which extend~\cite[Theorems~1.2 and~1.3]{savin_valdinoci_gamma_conv} to the setting under consideration here.
	
	\begin{theorem}[$\Gamma$-convergence of	the perturbed fractional Allen-Cahn energy functional] \label{th::gamma_conv}
		Let~$s\in(0,1)$. Then, the sequence of functionals~$\E_\epsilon$ $\Gamma$-converges to~$\E$, i.e. for any~$u\in X_M$ we have
		\begin{enumerate}[(i)]
			\item for any sequence~$\{u_\epsilon\}_\epsilon\subset X_M$ converging to~$u$ in~$L^1_{\loc}(\R^n)$,
			$$ \E(u,\Omega) \leq \liminf_{\epsilon\to0^+} \E_\epsilon(u_\epsilon,\Omega) ,$$
			\item there exists a sequence~$\{u_\epsilon\}_\epsilon\subset X_M$ converging to~$u$ in~$L^1_{\loc}(\R^n)$ such that 
			$$ \E(u,\Omega) \geq \limsup_{\epsilon\to0^+} \E_\epsilon(u_\epsilon,\Omega) .$$
		\end{enumerate}
	\end{theorem}
	
	\begin{theorem}[Limit stability of minimizers for the perturbed fractional Allen-Cahn energy functional]\label{th::min_conv}
		Let~$\{u_\epsilon\}_\epsilon\subset X_M$ be a sequence of minimizers for~$\E_\epsilon$ in~$\Omega$. Then, there exist a convergent subsequence and a set~$E\subseteq\R^n$ such that
		\begin{equation} \label{eq::u_conv}
			u_\epsilon\to u:=\chi_E-\chi_{E^c} \quad\text{in }L^1(\Omega).
		\end{equation}
		Moreover, 
		\begin{enumerate}[(i)]
			\item\label{enum::min_conv1} if~$s\in(0,1/2)$ and~$u_\epsilon$ is weak* convergent to some~$u_0$ in~$L^\infty(\Omega^c)$, then~$u$ minimizes~$\E$ among all the functions that coincide with~$u_0$ in~$\Omega^c$;
			\item\label{enum::min_conv2} if~$s\in[1/2,1)$, then~$u$ minimizes~$\E$. Also, for any~$K\comp\Omega$, it holds
			\begin{equation} \label{eq::energy_upper_bound}
				\limsup_{\epsilon\to0^+} \E_\epsilon(u_\epsilon,K) \leq c_\star\Per(E,\overline{K}) +\frac{1}{2\omega_{n-1}}\int_{K}(\chi_E-\chi_{E^c})H\, dx = c_\star\P_{\overline{K}}^{H/c_\star}(E),
			\end{equation}
			where~$c_\star$ is as in~\eqref{eq::F} and~$\P_{\overline{K}}^{ H/c_\star}$ is
			as in~\eqref{gufiwyfew7965tyhieshrdef2}.
		\end{enumerate}
	\end{theorem}
	
	We point out that if~$u_0=\chi_{E_0}-\chi_{E_0^c}$ a.e. in~$\Omega^c$, for some set~$E_0$,  then Theorem~\ref{th::min_conv}-(i) entails that, if~$s\in(0,1/2)$, $u$ minimizes
	\begin{equation*}
		\E(u,\Omega)=\K(u,\Omega)+\H(u,\Omega) =\Per_s(E,\Omega)+\frac1{2(1-2s)}\int_\Omega(\chi_E-\chi_{E^c}) Hdx=\P_{s, \Omega}^{H}(E), 
	\end{equation*}
	Namely, $u$ is a solution of the non-local Massari Problem with external datum~$E_0$	(recall the alternative formulation of the non-local Massari functional in~\eqref{gufiwyfew7965tyhieshrdef}).
	
	Similarly, if~$s\in[1/2,1)$, Theorem~\ref{th::min_conv}-(ii) gives that~$E$ is a minimizer for~$\mathscr{P}_\Omega^{H/c_\star}$, and hence it
	is a set of mean curvature~$H/c_\star$.
	
	The proof of Theorem~\ref{th::gamma_conv} is contained in Section~\ref{sec::proof_gamma_conv}. As for Theorem~\ref{th::min_conv}, we address the case~$s\in(0,1/2)$ in Section~\ref{sec::s_less_half}, whereas Section~\ref{sec::s_geq_half} is devoted to the case~$s\in[1/2,1)$.
	
	\subsection{A prescribed-mass Allen-Cahn Problem}
	The prescribed-mass Allen-Cahn Problem is an interesting example of non-linear partial differential equation subject to a constraint. This model comes from the gradient theory of phase transitions and has been deeply studied in the classical case in the 80s (see e.g.~\cite{modica_phase_transition1, modica_phase_transition2, luckhaus_modica}). 
	
	We present the fractional counterpart of the convergence results discussed there. In order to do this, let~$m\in[-|\Omega|,|\Omega|]$, and let us define the family 
	\begin{equation} \label{eq::setW}
		\begin{split}
			Z_{M,m}&:=\left\{u\in L^1_{\loc}(\R^n) \mbox{ s.t. } \norm{u}_{L^\infty(\Omega)}\leq M \mbox{, and }\int_\Omega u\,dx=m\right\} \\
			&\ = X_M\cap\left\{\int_\Omega u\,dx=m\right\}.
		\end{split}
	\end{equation}
	
	We consider the restriction of the energy functionals~$\F_\epsilon$ and~$ \F$ defined in~\eqref{eq::F_epsilon} and~\eqref{eq::F} to~$Z_{M,m}$, and address the minimization problem
	\begin{equation} \label{eq::prescribed_mass_AC}
		\text{find~$u_\epsilon\in Z_{M,m}$ such that }\F_\epsilon(u_\epsilon,\Omega) = \min\{\F_\epsilon(v,\Omega) \text{ s.t. }v\in Z_{M,m}\}.
	\end{equation}
	
	\begin{rem}
		We mention that in~\eqref{eq::prescribed_mass_AC} we do not prescribe any Dirichlet boundary condition. Indeed, this would result in an over-determined problem, for which one
		cannot expect the existence of a solution even in dimension one. We refer to Appendix~\ref{append::over_determined} for counter-examples both in the classical and in the non-local case.
	\end{rem}
	
	To start with, we claim that the sequence of functionals~$\F_\epsilon|_{Z_{M,m}}$ $\Gamma$-converges to~$\F|_{Z_{M,m}}$.
	
	\begin{theorem}[$\Gamma$-convergence of
		the fractional Allen-Cahn energy functional subject to a mass constraint]\label{th::gamma_conv_mass_prescribed}
		Let~$s\in(0,1)$. For any~$u\in Z_{M,m}$, the following holds:
		\begin{enumerate}[(i)]
			\item \label{item::liminf_mass_prescribed}for any sequence~$\{u_\epsilon\}_\epsilon\subset Z_{M,m}$ converging to~$u$ in~$L^1_{\loc}(\R^n)$, we have that
			\begin{equation*}
				\F(u,\Omega) \leq \liminf_{\epsilon\to0^+} \F_\epsilon(u_\epsilon,\Omega);
			\end{equation*}
			\item \label{item::limsup_mass_prescribed}there exists a sequence~$\{u_\epsilon\}_\epsilon\subset Z_{M,m}$ converging to~$u$ in~$L^1_{\loc}(\R^n)$ such that
			\begin{equation*}
				\F(u,\Omega) \geq \limsup_{\epsilon\to0^+} \F_\epsilon(u_\epsilon,\Omega) .
			\end{equation*}
		\end{enumerate}
	\end{theorem}
	
	Moreover, we have that a sequence of minimizers~$\{u_\epsilon\}_\epsilon$ of~$\F_\epsilon$ subject to a mass constraint has uniformly bounded energy. 
		
		\begin{theorem}[Energy bounds for the fractional Allen-Cahn energy functional subject to a mass constraint] \label{th::energy_uniform_estimate}
			Let~$\{u_\epsilon\}_\epsilon\subset Z_{M,m}$ be a sequence of minimizers of~$\F_\epsilon$ in~$\Omega$.
			
			Then,
			\begin{equation} \label{eq::energy_unif_estimate}
				\F_\epsilon(u_\epsilon,\Omega)\leq C,
			\end{equation}
			for some positive constant~$C$ independent of~$\epsilon$.
		\end{theorem}
	
	We point out that the mass constraint in~\eqref{eq::setW} is closed with respect to the~$L^1_{\loc}$-convergence, hence Theorem~\ref{th::gamma_conv_mass_prescribed}-\eqref{item::liminf_mass_prescribed} can be proved with the same argument as in~\cite[Theorem~1.2-(i)]{savin_valdinoci_gamma_conv}. 
	
	For what concern Theorem~\ref{th::gamma_conv_mass_prescribed}-\eqref{item::limsup_mass_prescribed} instead, we need to verify that the recovery sequence~$\{u_\epsilon\}_{\epsilon}$ remains in~$Z_{M,m}$ for every~$\epsilon$. 
	When~$s\in(0,1/2)$, as shown in~\cite[Theorem~1.2]{savin_valdinoci_gamma_conv}, one can take the sequence~$\{u_\epsilon=u\}_\epsilon$ and obtain the desired result.
	
	On the other hand, if~$s\in[1/2,1)$, we need a suitable 
	perturbation~$\phi_\epsilon$ such that~$\phi_\epsilon\to0$ pointwise, 
	\begin{align*}
		&\F(u,\Omega) \geq \limsup_{\epsilon\to0^+} \F_\epsilon(u_\epsilon+\phi_\epsilon,\Omega) ,\\
		\text{and}\quad&\int_\Omega( u_\epsilon + \phi_\epsilon) dx = m,
	\end{align*}
	where~$\{u_\epsilon\}_\epsilon$ is the sequence constructed in~\cite[Proposition~4.6]{savin_valdinoci_gamma_conv}.
	Full details of this argument are contained in Section~\ref{sec::mass_prescribed_AC},
	while a proof of Theorem~\ref{th::energy_uniform_estimate} can be found in Section~\ref{sec::energy_uniform_estimate}.
	\medskip
		
	Now, notice that if~$u_\epsilon$ is a solution of~\eqref{eq::prescribed_mass_AC}, then~$u_\epsilon$ is also a solution of the inhomogeneous fractional\footnote{
			Here and in the rest of the paper, the operator~$(-\Delta)^s$ is the fractional Laplacian, defined, for all~$s\in(0,1)$, as
			\begin{equation*}
				(-\Delta)^su(x):=2{\mbox{P.V.}} \int_{\R^n}\frac{u(x)-u(y)}{|x-y|^{n+2s}}dy:=2\lim_{\delta\to0} \int_{\R^n\setminus B_{\delta}(x)}\frac{u(x)-u(y)}{|x-y|^{n+2s}}dy,
			\end{equation*}
			where P.V. stands for the principal value notation.
		} Allen-Cahn equation
		\begin{equation}\label{eq::inhomog_mass_AC}
			\epsilon^{2s}(-\Delta)^su_\epsilon+W'(u_\epsilon)+\lambda_\epsilon=0,
		\end{equation}
		where~$\lambda_\epsilon$ is the Lagrange multiplier associated with the conservation law
		\begin{equation}\label{bfjdew9r7912345678kjhgf}
			\int_\Omega u_\epsilon dx=m,
		\end{equation}
		and subject to a non-local Neumann external condition (see~\eqref{eq::u_def_outside_omega} in Theorem~\ref{th::u_def_outside_omega}).
		
		In turn, we find that solutions of such an equation are unconstrained critical points of a perturbed energy functional $\G_\epsilon$ (see~\eqref{eq::def_G_epsilon}). 
		Moreover, thanks to Theorem~\ref{th::energy_uniform_estimate}, we have that a suitable rescaling of the Lagrange multipliers $\{\lambda_\epsilon\}_\epsilon$ is bounded uniformly in $\epsilon$. Thus, we infer that $\G_\epsilon$ is $\Gamma$-convergent to a limit functional $\G$, whose first variation is related with the non-local hybrid mean curvature (see Definition~\ref{def::hybrid_curv} below). 
		\medskip
	
	We start drawing attention to the fact that the mass constraint reads a non-local Neumann external condition, which forces a strict relationship between how any solution of~\eqref{eq::prescribed_mass_AC} behaves inside and outside~$\Omega$. More precisely, the following holds true:
	
	\begin{theorem}[External behavior of solutions of the prescribed-mass non-local Allen-Cahn Equation] \label{th::u_def_outside_omega}
		Let~$u_\epsilon\in Z_{M,m}$ be a solution of~\eqref{eq::prescribed_mass_AC}.
		
		Then,
		\begin{equation}  \label{eq::u_def_outside_omega}
			u_\epsilon(x) = \frac{\displaystyle\int_\Omega \frac{u_\epsilon(y)}{|x-y|^{n+2s}}dy }{\displaystyle\int_\Omega \frac{1}{|x-y|^{n+2s}}dy }\quad\text{a.e. in }\Omega^c.
		\end{equation}
	\end{theorem}
	
	\begin{rem} \label{rem::not_char_funct}
		The right-hand side of~\eqref{eq::u_def_outside_omega} corresponds to a non-local Neumann condition, see~\cite{MR3651008}.
		
		We also observe that, in general, a function~$u$ satisfying~\eqref{eq::u_def_outside_omega}
		is not a signed characteristic
		function, i.e. a function such that~$|u|=1$ almost everywhere in~$\R^n$. Indeed, notice that, for almost any~$x\in\Omega^c$, we have that~$u(x)=\pm1$ if and only if either~$E\cap\Omega=\Omega$ or~$E\cap\Omega=\varnothing$.
		
		This is an interesting structural difference with respect to the existing literature about
		non-local interfaces, which mostly focused on ``pure phases''
		identified by global signed characteristic
		functions.
	\end{rem}
	
	
	Next, we show that a suitable rescaling of the Lagrange multipliers sequence~$\{\lambda_\epsilon\}_\epsilon$ defined in~\eqref{eq::inhomog_mass_AC} converges (up to a subsequence) in~$\R$, as~$\epsilon\to0^+$.
	
	To this purpose, let us define the quantity
	\begin{equation}\label{kappaepsilondef}
		\kappa_\epsilon:=
		\begin{cases}
			\epsilon^{-2s},\quad&\text{if }s\in\left(0,\frac{1}{2}\right),\\
			|\epsilon\log\epsilon|^{-1},\quad&\text{if }s=\frac{1}{2},\\
			\epsilon^{-1},\quad&\text{if }s\in\left(\frac{1}{2},1\right).
		\end{cases}
	\end{equation}
	Then, we prove the following:
	
	\begin{theorem} [Convergence of the Lagrange multipliers for the fractional Allen-Cahn energy functional subject to a mass constraint]
		\label{th::multiplier_bound} 
		Let~$\{u_\epsilon\}_\epsilon$ be a sequence of solutions of~\eqref{eq::inhomog_mass_AC} and~\eqref{bfjdew9r7912345678kjhgf}
		with~$m\in(-|\Omega|,|\Omega|)$,
		such that~$u_\epsilon\to u:=\chi_E-\chi_{E^c}$ in~$L^1(\Omega)$, for some set~$E\subseteq\R^n$.
		
		Suppose that
		\begin{equation} \label{eq::energy_unif_estimateBIS}
			\F_\epsilon(u_\epsilon,\Omega)\leq C,
		\end{equation}
		for some positive constant~$C$ independent of~$\epsilon$.
		
		Then, the sequence~$\{\mu_\epsilon:=\kappa_\epsilon\lambda_\epsilon\}_\epsilon$ is bounded in~$\R$, and hence convergent up to a subsequence. 
	\end{theorem}
	
	In light of Theorem~\ref{th::multiplier_bound}, for every~$\epsilon>0$, we define the perturbed energy functional~$\G_\epsilon:X_M\to\R\cup\{+\infty\}$ associated with~\eqref{eq::inhomog_mass_AC} as
	\begin{equation} \label{eq::def_G_epsilon}
		\G_\epsilon(v,\Omega) := \F_\epsilon(v,\Omega)+ \mu_\epsilon\int_\Omega v\, dx.
	\end{equation}
	
	Moreover, assume that
	\begin{equation} \label{eq::mu_lagr_mult}
		\mu := \lim_{\epsilon\to0^+} \mu_\epsilon.
	\end{equation}
	Then, we deduce the following convergence result.
	
	\begin{theorem}[Energy limits subject to mass constraints] \label{th::energy_continuity}
		Let~$\{u_\epsilon\}_\epsilon\subset Z_{M,m}$  be a sequence of solutions for~\eqref{eq::prescribed_mass_AC}, i.e. minimizers for the energy functional~$\F_\epsilon$ subject to a mass constraint. 
		
		Then, there exist a convergent subsequence and a function~$u\in Z_{M,m}$ such that~$u_\epsilon\to u$ in~$L^1_{\loc}(\R^n)$, and~$u=\chi_E-\chi_{E^c}$ in~$\Omega$, for some set~$E$.
		
		Moreover, $u$ is a minimizer for~$\F$ in~$\Omega$ under the mass constraint and
		\begin{equation} \label{eq::F_convmass_prescribed}
			\lim_{\epsilon\to0^+} \F_\epsilon(u_\epsilon,\Omega) = \F(u,\Omega) .
		\end{equation}
		Besides,
		\begin{equation} \label{eq::G_convmass_prescribed}
			\lim_{\epsilon\to0^+} \G_\epsilon(u_\epsilon,\Omega) = \F(u,\Omega)+\mu\int_\Omega u\, dx =: \G(u,\Omega) .
		\end{equation}
	\end{theorem}
	
	\begin{rem}
		We stress that Theorem~\ref{th::energy_continuity} establishes the convergence of the sequence~$\{u_\epsilon\}_\epsilon$ in~$L^1_{\loc}(\R^n)$, rather than only in~$L^1(\Omega)$ (compare with Theorem~\ref{th::min_conv} above). This is a direct consequence of the non-local Neumann condition (see Theorem~\ref{th::u_def_outside_omega}).
	\end{rem}
	
	We prove Theorems~\ref{th::u_def_outside_omega} and~\ref{th::multiplier_bound} in Sections~\ref{sec::u_def_outside_omega} and~\ref{sec::lagr_mult_conv_prescribed_mass} respectively, 
	while Section~\ref{sec::min_conv_prescribed_mass} is devoted to the proof of Theorem~\ref{th::energy_continuity}.
	
	\subsection{Convergence for critical points of~$\G_\epsilon$ and non-local hybrid mean curvature}
	In~\cite{MR3900821}, the authors have proved convergence results for critical points of a fractional Allen-Cahn equation with Dirichlet boundary conditions and~$s\in(0,1/2)$. Here we generalize~\cite[Theorem~5.1]{MR3900821} relaxing the assumptions on the external data~$\{g_k\}_k$. Indeed, as shown in Theorem~\ref{th::energy_continuity}, a sequence~$\{u_\epsilon\}_\epsilon$ of minimizers for~$\F_\epsilon|_{Z_{M,m}}$ converges in~$L^1_{\loc}(\R^n)$ to a minimizer~$u$ of the limit functional~$\F|_{Z_{M,m}}$. However, one cannot expect~$u$ to be a signed characteristic function in the whole~$\R^n$, as pointed out by
	Theorem~\ref{th::u_def_outside_omega} and Remark~\ref{rem::not_char_funct}. 
	
	This fact leads us to introduce a new geometric object that, as far as we know, has not yet been investigated in the literature, and which coincides with the  first variation of the~$K$-contribution of the fractional~$H^s$-Gagliardo semi-norm of~$u$, for any Lipschitz~$K\comp\Omega$. Since~$u=\chi_E-\chi_{E^c}$ in~$\Omega$, for some set~$E$, we refer to this new object as the
	``non-local hybrid mean curvature'' of the couple~$(E,u|_{\Omega^c})$.
	
	More precisely:
	
	\begin{definition}[Non-local hybrid mean curvature]\label{def::hybrid_curv}
		Let~$s\in(0,1/2)$, and consider a bounded, Lipschitz, open set~$\Omega\subset\R^n$. Consider a set~$E$ and a function~$g\in L^\infty(\Omega^c)$, and let
		\begin{equation*}
			u:=
			\begin{cases}
				\chi_E-\chi_{E^c},\quad&\mbox{in }\Omega,\\		
				g,\quad&\mbox{in }\Omega^c.
			\end{cases}
		\end{equation*} 
		
		Given a vector field~$X\in\cont^1_c(\Omega,\R^n)$, if~$\{\phi_t\}_t$ is the flow generated by~$X$, we define
		$$ \delta\K(u,\Omega)[X] := \frac{d}{dt}\K(\phi_t(u),\Omega) \big|_{t=0},$$
		where~$\K(u,\Omega)$
		has been introduced in~\eqref{eq::def_K}.
		
		We call~$\delta\K(u,\Omega)$ the \textbf{non-local hybrid mean curvature} of the couple~$(E,g)$.
	\end{definition}
	
	With this, we now state a version of~\cite[Theorem~5.1]{MR3900821} more suitable to our setting.
	
	\begin{theorem}[Convergence for solutions of inhomogeneous Allen-Cahn Equation] \label{th::millot_sire_wang}
		Let~$\Omega\subset\R^n$ be a Lipschitz, bounded, open set. Let~$\eta>0$, and define 
		$$ \Omega_\eta:=\{x\in\R^n \text{ s.t. } {\rm{dist}}(x,\Omega)<\eta\}.$$
		Also, let~$\{\epsilon_k\}_k\subset\R$ be a positive infinitesimal sequence.
		
		For all~$k\in\N$, let~$g_k:\Omega^c\to\R$. Suppose that~$\{g_k\}_k
		\subseteq\cont^{0,1}(\Omega_\eta\setminus \Omega)$  satisfy$$\sup_k\norm{g_k}_{L^\infty(\Omega^c)}<+\infty,$$ and~$g_k\to g$ in~$L^1_{\loc}(\Omega^c)$, for some function~$g$.
		
		Let~$\{f_k\}_k\subseteq\cont^{0,1}(\Omega)$ satisfy
		\begin{equation*}\label{eq::f_grow}
			\sup_k \Big(\epsilon_k^{2s}\norm{f_k}_{L^\infty(\Omega)}+\norm{f_k}_{W^{1,q}(\Omega)}\Big)<+\infty,\quad \text{ for some }q\in\left(\frac{n}{1-2s},n\right),
		\end{equation*}
		and be such that~$f_k\rightharpoonup f$ weakly in~$W^{1,q}(\Omega)$, for some function~$f\in W^{1,q}(\Omega)$.
		
		Let~$\{u_k\}_k$ be such that~$u_k\in H^s(\Omega)\cap L^p(\Omega)$, for some~$p\geq1$. Suppose that~$u_k$ is a weak solution of
		\begin{equation} \label{eq::weak_eq_u_k}
			\begin{cases}
				(-\Delta)^s u_k+\epsilon_k^{-2s}W'(u_k)=f_k\quad&\text{in }\Omega,\\
				u_k=g_k\quad&\text{in }\Omega^c,
			\end{cases}
		\end{equation}
		for every~$k\in\N$.
		
		If
		\begin{equation}\label{mnbvcxz21345678jhgfdsurieowjdh}\sup_k\left(\F_{\epsilon_k}(u_k,\Omega)-\int_{\Omega}f_ku_kdx\right)<+\infty,\end{equation} then there exist a convergent subsequence and a function~$u$ such that~$u_k\to u$ in~$L^1_{\loc}(\R^n)$, $u=\chi_E-\chi_{E^c}$ in~$\Omega$, for some set~$E$, and 
		\begin{equation*}
			\delta\K(u,\Omega)[X] = \kappa\int_{E\cap\Omega} \mbox{div}(fX)dx \quad\text{for every }X\in\cont^1_c(\Omega,\R^n), 
		\end{equation*}
		for some~$\kappa>0$ depending only on~$n$ and~$s$.
	\end{theorem}
	
	\begin{rem} Theorem~\ref{th::millot_sire_wang} is a refinement
		of~\cite[Theorem~5.1]{MR3900821}. This refinement is needed since, in our setting, one cannot expect a solution~$u_\epsilon$ of~\eqref{eq::prescribed_mass_AC} to be in~$\cont^{0,1}_{\loc}(\R^n)$
		(and therefore the statement of~\cite[Theorem~5.1]{MR3900821}, which relies
		on such an assumption, would not suffice for our purposes). 
		
		Indeed, by Formula~\eqref{eq::u_def_outside_omega}
		in Theorem~\ref{th::u_def_outside_omega}, we have that~$u_\epsilon\in\cont^\infty(\Omega^c) $, and
		\begin{equation} \label{eq::u_dervative_Omega_c}
			\begin{split}
				\partial_ju_\epsilon(x) &= \frac{(n+2s)}{\displaystyle\left(	\int_\Omega \frac{dy}{|x-y|^{n+2s}} \right)^2} \int_\Omega \frac{u_\epsilon(y)}{|x-y|^{n+2s}}dy\int_\Omega \frac{(x_j-y_j)}{|x-y|^{n+2s+2}}dy\\
				&\qquad- \frac{(n+2s)x_j}{\displaystyle\int_\Omega \frac{dy}{|x-y|^{n+2s}}}\int_\Omega \frac{u_\epsilon(y)(x_j-y_j)}{|x-y|^{n+2s+2}}dy,
			\end{split}
		\end{equation}
		for a.e.~$x\in\Omega^c$, for every~$j\in\{1,\dots,n\}$.
		
		However, \eqref{eq::u_dervative_Omega_c} may be not well defined when~$x$ approaches~$\partial\Omega$.
	\end{rem}
	
	As a  consequence of Theorem~\ref{th::millot_sire_wang}, we obtain that a sequence~$\{u_\epsilon\}_\epsilon$ of critical points for~$\G_\epsilon$ converges (up to a subsequence) to a critical point~$u$ of the limit functional~$\G$ in~$L^1_{\loc}(\R^n)$. Moreover, $u=\chi_E-\chi_{E^c}$ in~$\Omega$, for some set~$E$, and the couple~$(E,u|_{\Omega^c})$ has constant non-local hybrid mean curvature. 
	
	\begin{theorem}[Hybrid mean curvature of the phase interface under a mass constraint] \label{th::conv_crit_points}
		Let~$\{u_\epsilon\}_\epsilon$ be a sequence in~$Z_{M,m}$ such that~$u_\epsilon$ is
		a critical point for~$\G_\epsilon$ in~$\Omega$, i.e.~$u_\epsilon$ is a weak solution for~\eqref{eq::inhomog_mass_AC}.
		
		Then, there exist a convergent subsequence and a function~$u$ and a set~$E$ such that~$u_\epsilon\to u$ in~$L^1_{\loc}(\R^n)$, $u=\chi_E-\chi_{E^c}$ in~$\Omega$, and~$u$ is a critical point for the functional~$\G$ in~$K$, for any Lipschitz set~$K\comp\Omega$.
		
		In particular, the couple~$(E,u|_{\Omega^c})$ has non-local hybrid mean curvature~$\kappa\mu$ in~$K$, where~$\kappa$ is as in Theorem~\ref{th::millot_sire_wang}, and~$\mu$ is as in~\eqref{eq::mu_lagr_mult}.
	\end{theorem}
	
	\begin{rem} Notice that, in light of~\eqref{eq::u_def_outside_omega}, the function~$u$ coming from Theorem~\ref{th::conv_crit_points} is continuous across the boundary of~$\Omega$. Indeed, for every~$x_0\in E\cap\partial\Omega$,
		$$\lim_{\substack{x\to x_0 \\ x\in E\cap \Omega^c}}\frac{\displaystyle\int_\Omega\frac{\chi_E(y)-\chi_{E^c}(y)}{|x-y|^{n+2s}dy}}{\displaystyle\int_\Omega\frac{dy}{|x-y|^{n+2s}}} 
		= \lim_{\substack{x\to x_0 \\ x\in E\cap \Omega^c}} \frac{\displaystyle\int_{\Omega\cap E}\frac{dy}{|x-y|^{n+2s}}-\int_{\Omega\cap E^c}\frac{dy}{|x-y|^{n+2s}}}{\displaystyle\int_{\Omega\cap E}\frac{dy}{|x-y|^{n+2s}}+\int_{\Omega\cap E^c}\frac{dy}{|x-y|^{n+2s}}} = 1,$$
		and for every~$x_0\in E^c\cap\partial\Omega$,
		$$\lim_{\substack{x\to x_0 \\ x\in E^c\cap \Omega^c}}\frac{\displaystyle\int_\Omega\frac{\chi_E(y)-\chi_{E^c}(y)}{|x-y|^{n+2s}}dy}{\displaystyle\int_\Omega\frac{dy}{|x-y|^{n+2s}}} 
		= \lim_{\substack{x\to x_0 \\ x\in E^c\cap \Omega^c}} \frac{\displaystyle\int_{\Omega\cap E}\frac{dy}{|x-y|^{n+2s}}-\int_{\Omega\cap E^c}\frac{dy}{|x-y|^{n+2s}}}{\displaystyle\int_{\Omega\cap E}\frac{dy}{|x-y|^{n+2s}}+\int_{\Omega\cap E^c}\frac{dy}{|x-y|^{n+2s}}} = -1.$$
		On the other hand, we infer from~\eqref{eq::u_dervative_Omega_c} that~$u$ is not~$\cont^1$ across~$\partial\Omega$.
		
		Therefore, non-local phase interfaces subject to a mass constraint show 
		a boundary behavior of type ``$\cont^0$  but not~$\cont^1$''.
		
		This boundary behavior 
		is very different from that of~$s$-minimal graphs, which can only be either discontinuous or~$\cont^1$ at the boundary according to~\cite[Theorem~1.2]{MR4104542}.
	\end{rem}
	
	The classical analog of Theorem~\ref{th::conv_crit_points} has been presented in~\cite{luckhaus_modica}, where the authors proved a correspondence between the mean curvature of the limit set~$E$ and the limit of the Lagrange multiplier associated with the mass constraint.
	
	However, to the best of our knowledge, no result on the convergence for critical points of the non-local Allen-Cahn Energy is yet known when~$s\geq1/2$.
	
	The proof of Theorem~\ref{th::millot_sire_wang} is contained in Section~\ref{sec::millot_sire_wang}. Then, we conclude our discussion proving Theorem~\ref{th::conv_crit_points} in Section~\ref{sec::conv_crit_points}.

	
	\section[Proof of Theorem~\ref{th::gamma_conv_J}	\\ {\footnotesize$\Gamma$-convergence of the non-local Massari Problem}]{Proof of Theorem~\ref{th::gamma_conv_J} - $\Gamma$-convergence of the non-local Massari Problem}\label{sec::proof_gamma_conv_J}
	
	This section is devoted to the proof of Theorem~\ref{th::gamma_conv_J}. This proof is a direct consequence of the forthcoming Proposition~\ref{prop::L2_cont_conv} and~\cite[Theorem~2]{ambrosio_dephilippis_martinazzi}.
	
	\begin{prop} \label{prop::L2_cont_conv}
		Let~$\{f_j\}_j$, $\{g_j\}_j\subset L^\infty(\R^n)$ be such that 
		\begin{equation*}
			\sup_j\norm{f_j}_{L^\infty(\R^n)}<+\infty\qquad
			\text{and}\qquad\sup_j\norm{g_j}_{L^\infty(\R^n)}<+\infty.
		\end{equation*}
		Suppose that there exist functions~$f$, $g$ such that~$f_j\to f$ and~$g_j\to g$ in~$L^1_{\loc}(\R^n)$ as~$j\to+\infty$. 
		
		Then,
		\begin{equation} \label{eq::L2_cont_conv}
			\lim_{j\to+\infty}\int_K f_j g_j dx = \int_K f\,g\,dx,
		\end{equation}
		for every compact set~$K\subset\R^n$.
	\end{prop}
	
	\begin{proof}
		Let~$K\subset\R^n$ be a compact set, and let~$f_j$, $g_j$, $f$, and~$g$ be as in the assumptions. First, notice that~$f$, $g\in L^\infty(\R^n)$.
		
		Then, the claim follows immediately from the estimate
		\begin{equation*}
			\left|\int_K (f_jg_j-fg)dx\right| \leq \norm{f}_{L^\infty(K)}\norm{g-g_j}_{L^1(K)} + \sup_j\norm{g_j}_{L^\infty(K)}\norm{f-f_j}_{L^1(K)}.\qedhere
		\end{equation*}
	\end{proof}
	
	As an easy consequence of Proposition~\ref{prop::L2_cont_conv}, we obtain that if~$H_j\to H$ in~$L^1(\Omega)$, as in Theorem~\ref{th::gamma_conv_J}, then
	\begin{equation*}
		\lim_{j\to+\infty} \int_{\Omega\cap F} H_jdx =  \int_{\Omega\cap F} H\,dx,
	\end{equation*}
	for every set~$F\subseteq\R^n$.	Thus, in light of~\cite[Remark~4.8]{dal_maso}, we infer the following:
	
	\begin{cor} \label{cor::H_cont_conv}
		For every set~$F$, define the functionals
		\begin{align*}
			&\mathscr{H}_j(F) := \int_{\Omega\cap F} H_jdx,\quad \text{for all } j\in\N,\\
			\text{and}\quad &\mathscr{H}(F) := \int_{\Omega\cap F} H\,dx.
		\end{align*}
		
		Then, $\mathscr{H}_j$ is pointwise convergent in~$L^1_{\loc}(\R^n)$ to~$\mathscr{H}$ in the sense of~\cite[Definition~4.7]{dal_maso}. 
	\end{cor}
	
	Finally, Corollary~\ref{cor::H_cont_conv}, together with~\cite[Theorem~2]{ambrosio_dephilippis_martinazzi} and~\cite[Theorem~6.20]{dal_maso}, entails Theorem~\ref{th::gamma_conv_J}.
	
	
	\section[Proof of Theorem~\ref{th::min_J_conv}	\\ {\footnotesize Energy bounds and limit stability of minimizers for the non-local Massari Problem}]{Proof of Theorem~\ref{th::min_J_conv}	- Energy bounds and limit stability of minimizers for the non-local Massari Problem}\label{sec::proof_min_J_conv}
	
	In this section, we present the proof of Theorem~\ref{th::min_J_conv} about the convergence of minimizers of the Massari functionals~$\{\mathscr{J}_{s_j,\Omega}^{H_j}\}_j$. To do this, we follow some ideas introduced in~\cite[Section~4]{ambrosio_dephilippis_martinazzi}, also taking into account the extra term related to the non-local mean curvature.
	
	To lighten the notation, we define, for every set~$F$, the functionals
	\begin{equation*}
		\begin{split}
			& \mathscr{S}^1_s(F,\Omega) := \mathcal{L}(F\cap\Omega,F^c\cap\Omega)\\
			\text{and }\quad &\mathscr{S}^2_s(F,\Omega) := \mathcal{L}(F\cap\Omega,F^c\cap\Omega^c)+ \mathcal{L}(F\cap\Omega^c,F^c\cap\Omega),
		\end{split}
	\end{equation*}
	so that 
	\begin{equation*}
		\Per_s(F,\Omega) = \mathscr{S}^1_s(F,\Omega) + \mathscr{S}^2_s(F,\Omega).
	\end{equation*}
	
	\begin{proof} [Proof of Theorem~\ref{th::min_J_conv}]
		To prove~\eqref{eq::J_uniform_bound}, let us define, for any~$\delta>0$, the set
		\begin{equation*}
			\Omega_{-\delta} := \{x\in\Omega \text{ s.t. } {\rm{dist}}(x,\partial \Omega) <\delta \}
		\end{equation*}
		and consider~$F_j:=E_j\cap(\Omega^c\cup\Omega_{-\delta})$.
		
		We observe that
		\begin{equation*}\begin{split}&
				\Per_{s_j}(E_j,\Omega)-	\Per_{s_j}(E_j,\Omega\setminus\Omega_{-\delta} )
				\\&\quad=\mathcal{L}(E_j\cap\Omega_{-\delta},E_j^c\cap\Omega_{-\delta})
				+\mathcal{L}(E_j\cap\Omega^c,E_j^c\cap\Omega_{-\delta})
				\ge \mathscr{S}^1_{s_j}(E_j,\Omega_{-\delta} ).
		\end{split}\end{equation*}		
		Also, we recall from~\cite{MR5015001} that if~$E_j$ is minimizer for the functional~$\mathscr{J}_{s_j,\Omega}^{H_j}$, then it is a~$\Lambda_j$-minimal set, with~$\Lambda_j:=\norm{H_j}_{L^\infty}$. Thus, we have
		\begin{equation}\label{fhejasfj098765vdfkvl09876}
			\begin{split}
				&\limsup_{j\to+\infty} (1-2s_j)\Per_{s_j}(E_j,\Omega\setminus\Omega_{-\delta} ) \\
				& \qquad\leq \limsup_{j\to+\infty} (1-2s_j)\left(\Per_{s_j}(E_j,\Omega)-\mathscr{S}^1_{s_j}(E_j,\Omega_{-\delta} )\right) \\
				& \qquad\leq  \limsup_{j\to+\infty} (1-2s_j)\left(\Per_{s_j}(F_j,\Omega)+ \Lambda_j|\Omega_{-\delta}|-\mathscr{S}^1_{s_j}(E_j,\Omega_{-\delta} )\right) \\
				& \qquad=  \limsup_{j\to+\infty} (1-2s_j)\left(\mathscr{S}^1_{s_j}(F_j,\Omega)+\mathscr{S}^2_{s_j}(F_j,\Omega)-\mathscr{S}^1_{s_j}(E_j,\Omega_{-\delta} )\right) + (1-2s_j)\Lambda_j|\Omega_{-\delta}| . 
			\end{split}
		\end{equation}
		
		Now, since~$F_j\cap(\Omega\setminus\Omega_{-\delta})=\varnothing$, we see that
		\begin{eqnarray*}
			&&\mathscr{S}^1_{s_j}(F_j,\Omega)-\mathscr{S}^1_{s_j}(E_j,\Omega_{-\delta} )\\&=&
			\mathcal{L}(F_j\cap\Omega,F_j^c\cap\Omega)
			-\mathcal{L}(E_j\cap\Omega_{-\delta},E_j^c\cap\Omega_{-\delta})\\&=&
			\mathcal{L}(E_j\cap\Omega_{-\delta}, \Omega\setminus\Omega_{-\delta})
			+\mathcal{L}(E_j\cap\Omega_{-\delta},E_j^c\cap\Omega_{-\delta})
			-\mathcal{L}(E_j\cap\Omega_{-\delta},E_j^c\cap\Omega_{-\delta})\\&=&
			\mathcal{L}(E_j\cap\Omega_{-\delta}, \Omega\setminus\Omega_{-\delta})\\
			&\le&\mathcal{L}(\Omega_{-\delta}, \Omega\setminus\Omega_{-\delta}),
		\end{eqnarray*}	
		and therefore we can exploit~\cite[Proposition~16]{ambrosio_dephilippis_martinazzi}
		and obtain that
		\begin{equation} \label{eq::limsup_bound_1}\begin{split}
				\limsup_{j\to+\infty} (1-2s_j)\left(\mathscr{S}^1_{s_j}(F_j,\Omega)-\mathscr{S}^1_{s_j}(E_j,\Omega_{-\delta} )\right)& \leq 
				\limsup_{j\to+\infty} (1-2s_j)\mathcal{L}(\Omega_{-\delta}, \Omega\setminus\Omega_{-\delta})\\&\le
				\frac{n\omega_n}{2}\Per(\Omega\setminus\Omega_{-\delta},\R^n) .\end{split}\end{equation}
		
		Similarly,
		\begin{eqnarray*}
			\mathscr{S}^2_{s_j}(F_j,\Omega)
			&=&\mathcal{L}(F_j\cap\Omega,F^c_j\cap\Omega^c)+ \mathcal{L}(F_j\cap\Omega^c,F_j^c\cap\Omega)\\
			&\leq&2\mathcal{L}(\Omega, \Omega^c)
		\end{eqnarray*} and thus
		\begin{equation}\label{976543vdshkasfgwegui}
			\limsup_{j\to+\infty} (1-2s_j)\mathscr{S}^2_{s_j}(F_j,\Omega) \leq n\omega_n\Per(\Omega,\R^n) .
		\end{equation}
		
		Moreover, we know that
		\begin{equation*}
			\sup_j \Lambda_j=\sup_{j} \norm{H_j}_{L^\infty(\Omega)} <+\infty.
		\end{equation*}
		Using this information, \eqref{eq::limsup_bound_1} and~\eqref{976543vdshkasfgwegui}
		into~\eqref{fhejasfj098765vdfkvl09876} we thereby obtain that
		\begin{equation*}
			\limsup_{j\to+\infty} (1-2s_j)\Per_{s_j}(E_j,\Omega\setminus\Omega_{-\delta} )\le
			\frac{n\omega_n}{2}\Per(\Omega\setminus\Omega_{-\delta},\R^n) +n\omega_n\Per(\Omega,\R^n) .	
		\end{equation*}
		As a consequence, for every~$K\comp\Omega$,
		\begin{equation*}
			\begin{split}
				&\limsup_{j\to+\infty} (1-2s_j)\mathscr{J}_{s_j,K}^{H_j}(E_j) \\
				&\quad\leq \limsup_{j\to+\infty} (1-2s_j)\Per_{s_j}(E_j,K) + \limsup_{j\to+\infty} \int_{K\cap E_j} H_j dx
				\leq C_1  +\norm{H}_{L^1(\Omega)},
			\end{split}
		\end{equation*}
		for some positive constant~$C_1$ depending only on~$n$
		and~$\Omega$, proving~\eqref{eq::J_uniform_bound}.
		
		Now, suppose that~$K\comp\Omega$ is such that~$\Per(E,\partial K)=0$, where~$E$ is as in the assumptions. Let us also adopt the notation
		\begin{equation*}
			\mathcal{M}^n := \{D\subseteq\R^n\text{such that~$D$ is Lebesgue-measurable}\}.
		\end{equation*}
		Then, for every set~$D\in\mathcal{M}^n$ and, for all~$j\in\N$,
		we define the function~$\pi_j:\mathcal{M}^n\to\R$ as
		\begin{equation*}
			\pi_j(D) := (1-2s_j)\, \mathscr{S}^1_{s_j}(E_j,D).
		\end{equation*}	
		Thanks to~\cite[Theorem~21]{ambrosio_dephilippis_martinazzi}, the sequence~$\{\pi_j\}_j$ weakly converges\footnote{For more details about weak convergence and other properties of these kinds of
			set functions (i.e. defined on~$\mathcal{M}^n$), we refer to the Appendix in~\cite{ambrosio_dephilippis_martinazzi}.}
		to a regular, monotone, and super-additive function~$\pi:\mathcal{M}^n\to\R$. 
		
		Now, we focus on proving~\eqref{eq::J_cont_conv_min}. Let us define, for every~$\delta>0$, the set
		\begin{equation*}
			K_\delta := \{x\in\R^n \text{ s.t. } {\rm{dist}}(x,K)<\delta\}.
		\end{equation*}
		Moreover, consider a set~$F$ such that~$E\Delta F\subset K$. Thus, there exists~$\delta>0$ such that~$E\Delta F\comp K_\delta \comp \Omega$. Besides, by~\cite[Theorem~2]{ambrosio_dephilippis_martinazzi}, there exists a sequence of sets~$\{F_j\}_j$ such that
		\begin{equation} \label{eq::exact_conv}
			\begin{split}
				&\lim_{j\to+\infty} \left| (F_j\Delta F)\cap K_\delta \right| = 0\\
				\text{and } \quad&\lim_{j\to+\infty} (1-2s_j)\Per_{s_j}(F_j,K_\delta) = \omega_{n-1}\Per(F,K_\delta).
			\end{split}
		\end{equation}
		Furthermore, thanks to~\cite[Proposition~11]{ambrosio_dephilippis_martinazzi}, for given~$0<\rho<\tau<\delta$, we find sets~$G_j\subseteq\R^n$ such that
		\begin{equation}\label{degfytruytjgnv564832}
			G_j :=
			\begin{cases}
				F_j \quad &\text{in }K_\rho,\\
				E_j \quad &\text{in }K_\tau^c,
			\end{cases}
		\end{equation}
		and with the property that, for every~$\epsilon>0$, 
		\begin{equation} \label{eq::prop_Gj}
			\begin{split}
				\mathscr{S}^1_{s_j}(G_j,K_\delta)
				&\leq \mathscr{S}^1_{s_j}(F_j,K_\delta)+ \mathscr{S}^1_{s_j}(E_j,K_\delta\setminus K_{\rho-\epsilon})  + \frac{C_2}{\epsilon^{n+2s_j}}\\
				&\quad +\frac{C_2}{1-2s_j} \left| (E_j\Delta F_j) \cap (K_\tau\setminus K_\rho)\right| + C_3 \left| (E_j\Delta F_j) \cap K_\delta \right|  ,
			\end{split}		
		\end{equation}
		for some constant~$C_2>0$ independent of~$j$.
		
		We now point out that
		\begin{equation}\label{uweiofgewuo876543}
			\mathscr{S}^2_{s_j}(G_j,K_\delta) = \mathcal{L}(K_\delta\cap G_j,K_\delta^c\cap G_j^c) + \mathcal{L}(K_\delta\cap G_j^c,K_\delta^c\cap G_j).
		\end{equation}
		Also, for every~$\sigma\in(\delta,{\rm{diam}}(\Omega))$
		we have that
		\begin{equation}  \label{eq::estimateS2_1}
			\begin{split}
				&\mathcal{L}(K_\delta\cap G_j,K_\delta^c\cap G_j^c) 
				= \mathcal{L}(K_\tau\cap G_j,K_\delta^c\cap E_j^c) + \mathcal{L}((K_\delta\setminus K_\tau)\cap G_j,K_\delta^c\cap E_j^c)\\
				&\qquad\leq \frac{C_3}{s_j(\delta-\tau)^{2s_j}}|G_j\cap K_\tau| + \mathcal{L}((K_\delta\setminus K_\tau)\cap E_j,K_\sigma^c\cap E_j^c) \\
				&\qquad\qquad + \mathcal{L}((K_\delta\setminus K_\tau)\cap E_j,(K_\sigma\setminus K_\delta)\cap E_j^c) \\
				&\qquad\leq \mathscr{S}^1_{s_j}(E_j,K_\sigma\setminus K_\tau)  + \frac{C_3}{s_j}\left(\frac{1}{(\delta-\tau)^{2s_j}} + \frac{1}{(\sigma-\delta)^{2s_j}}\right),
			\end{split}
		\end{equation}
		for some positive constant~$C_3=C_3(n,K,\tau,\delta)$ possibly changing from line to line.
		
		Similarly, we obtain the estimate 
		\begin{equation}  \label{eq::estimateS2_2}
			\begin{split}
				&\mathcal{L}(K_\delta\cap G_j^c,K_\delta^c\cap G_j)
				= \mathcal{L}(K_\tau\cap G_j^c,K_\delta^c\cap E_j) + \mathcal{L}((K_\delta\setminus K_\tau)\cap G_j^c,K_\delta^c\cap G_j)\\
				&\qquad\leq \frac{C_3}{s_j(\delta-\tau)^{2s_j}}|G_j^c\cap K_\tau| + \mathcal{L}((K_\delta\setminus K_\tau)\cap E_j^c,K_\sigma^c\cap E_j) \\
				&\qquad\qquad + \mathcal{L}((K_\delta\setminus K_\tau)\cap E_j^c,(K_\sigma\setminus K_\delta)\cap E_j) \\
				&\qquad\leq \mathscr{S}^1_{s_j}(E_j,K_\sigma\setminus K_\tau)  + \frac{C_3}{s_j}\left(\frac{1}{(\delta-\tau)^{2s_j}} + \frac{1}{(\sigma-\delta)^{2s_j}}\right).
			\end{split}
		\end{equation}
		Therefore, we are led by~\eqref{uweiofgewuo876543}, 
		\eqref{eq::estimateS2_1} and~\eqref{eq::estimateS2_2} to
		\begin{equation*}
			\begin{split}
				\mathscr{S}^2_{s_j}(G_j,K_\delta) \leq 2 \mathscr{S}^1_{s_j}(E_j,K_\sigma\setminus K_\tau)  + \frac{C_3}{s_j}\left(\frac{1}{(\delta-\tau)^{2s_j}} + \frac{1}{(\sigma-\delta)^{2s_j}}\right),
			\end{split}
		\end{equation*}
		
		It thereby follows that
		\begin{equation*}
			\limsup_{j\to+\infty} (1-2s_j)\mathscr{S}^2_{s_j}(G_j,K_\delta) \leq \limsup_{j\to+\infty} 2(1-2s_j) \mathscr{S}^1_{s_j}(E_j,K_\sigma\setminus K_\tau) .
		\end{equation*}
		{F}rom this and~\eqref{eq::prop_Gj}, we thus obtain that
		\begin{equation}\label{ghewhfsbdvfveuwiyt743694369}
			\begin{split}
				&\liminf_{j\to+\infty} (1-2s_j) \mathscr{J}_{s_j,K_\delta}^{H_j}(G_j) \\
				\leq\;& \liminf_{j\to+\infty} (1-2s_j) \mathscr{S}^1_{s_j}(G_j,K_\delta) + \limsup_{j\to+\infty} (1-2s_j) \mathscr{S}^2_{s_j}(G_j,K_\delta) + \lim_{j\to+\infty} \int_{K_\delta\cap G_j} H_j dx \\
				\le\;&\liminf_{j\to+\infty} (1-2s_j)\left(\mathscr{S}^1_{s_j}(F_j,K_\delta)+ \mathscr{S}^1_{s_j}(E_j,K_\delta\setminus K_{\rho-\epsilon})  + \frac{C_2}{\epsilon^{n+2s_j}}\right.\\
				&\qquad\qquad\qquad \left.+\frac{C_2}{1-2s_j} \left| (E_j\Delta F_j) \cap (K_\tau\setminus K_\rho)\right| + C_3 \left| (E_j\Delta F_j) \cap K_\delta \right| \right)\\&\quad
				+\limsup_{j\to+\infty} 2(1-2s_j) \mathscr{S}^1_{s_j}(E_j,K_\sigma\setminus K_\tau) + \lim_{j\to+\infty} \int_{K_\delta\cap G_j} H_j dx	\\
				\leq \;&\liminf_{j\to+\infty} (1-2s_j) \mathscr{S}^1_{s_j}(F_j,K_\delta)+ 3\limsup_{j\to+\infty} (1-2s_j)\mathscr{S}^1_{s_j}(E_j,K_\sigma\setminus K_{\rho-\epsilon}) \\
				&\quad +C_2 \lim_{j\to+\infty} \left| (E_j\Delta F_j)\cap (K_\tau\setminus K_\rho)\right| \\&\quad+ \lim_{j\to+\infty} \left(\int_{K_\rho\cap F_j} H_j dx + \norm{H_j}_{L^\infty(\Omega)}|K_\delta\setminus K_\rho|\right).
		\end{split}\end{equation}
		
		Now, we observe that, since~$E\Delta F \subset K$, using also that~$E_j\to E$ in~$L^1_{\loc}(\R^n)$ and~\eqref{eq::exact_conv}, we have that 
		\begin{equation*}
			\lim_{j\to+\infty} \left| (E_j\Delta F_j)\cap (K_\tau\setminus K_\rho)\right|  = 0.
		\end{equation*}	
		Moreover, according to~\cite[Proposition~22]{ambrosio_dephilippis_martinazzi}, 
		\begin{equation*}
			\lim_{\substack{\sigma,\rho\to\delta\\ \epsilon\to0}} \limsup_{j\to+\infty} (1-2s_j)\mathscr{S}^1_{s_j}(E_j,K_\sigma\setminus K_{\rho-\epsilon}) = \lim_{\epsilon\to0}\limsup_{j\to+\infty} \pi_j(K_{\delta+\epsilon}\setminus K_{\delta-\epsilon}) = 0 .
		\end{equation*}
		Besides, since the convergence of~$H_j$ to~$H$ is uniform in~$\rho$, from Proposition~\ref{prop::L2_cont_conv} and~\eqref{eq::exact_conv}
		we also deduce that
		\begin{equation*}
			\lim_{\rho\to\delta} \lim_{j\to+\infty} \left[\int_{K_\rho\cap F_j} H_j dx + \norm{H_j}_{L^\infty(\Omega)}|K_\delta\setminus K_\rho|\right] =  \lim_{j\to+\infty} \int_{K_\delta\cap F_j} H_jdx = \int_{K_\delta\cap F} H\,dx.
		\end{equation*}
		Using these pieces of information into~\eqref{ghewhfsbdvfveuwiyt743694369}, and recalling the second limit in~\eqref{eq::exact_conv}, we thus find that	
		\begin{equation}\label{gyrdERTYbkher743dkgtvf76543}\begin{split}
				\liminf_{j\to+\infty} (1-2s_j) \mathscr{J}_{s_j,K_\delta}^{H_j}(G_j)& \leq \lim_{j\to+\infty} \left((1-2s_j) \mathscr{S}^1_{s_j}(F_j,K_\delta) + \int_{K_\delta\cap F_j} H_j dx\right) \\&=\lim_{j\to+\infty} (1-2s_j) \mathscr{S}^1_{s_j}(F_j,K_\delta) + \int_{K_\delta\cap F} H dx \\&=\omega_{n-1}\mathscr{J}_{K_\delta}^H(F) .
		\end{split}\end{equation}
		
		Now we point out that~$G_j$ is a competitor for~$E_j$ (recall the definition in~\eqref{degfytruytjgnv564832}), and therefore, by~\eqref{gyrdERTYbkher743dkgtvf76543} and the minimality of~$E_j$,
		\begin{equation*}
			\liminf_{j\to+\infty} (1-2s_j) \mathscr{J}_{s_j,K_\delta}^{H_j}(E_j)  \leq \omega_{n-1}\mathscr{J}_{K_\delta}^H(F) .
		\end{equation*}
		{F}rom this and Theorem~\ref{th::gamma_conv_J}-(ii), we thereby conclude that
		\begin{equation}\label{eq::E_min_J}
			\omega_{n-1}\mathscr{J}_{K_\delta}^H(E) \le \liminf_{j\to+\infty} (1-2s_j) \mathscr{J}_{s_j,K_\delta}^{H_j}(E_j) \leq\omega_{n-1}\mathscr{J}_{K_\delta}^H(F).\end{equation}
		Hence, since~$E\Delta F\subset K \comp K_\delta$, we infer that 
		\begin{equation*}
			\mathscr{J}_{K}^H(E) \leq \mathscr{J}_{K}^H(F) ,
		\end{equation*} which proves that~$E$ is a local minimizer for the classical Massari functional.
		
		To conclude, notice that, choosing~$F:=E$ in~\eqref{eq::E_min_J}, we obtain~\eqref{eq::J_cont_conv_min}, as desired.
	\end{proof}
	
	\begin{rem}
		Though we will not use this fact in this paper, we point out that a small modification in the argument used for the proof of Theorem~\ref{th::min_J_conv} provides
		a different proof of~\cite[Proposition~2.1]{MR5015001}, which establishes the convergence of a sequence of almost minimizers of the fractional perimeter
		to an almost minimizer.
	\end{rem}
	
	\section[Proof of Theorem~\ref{th::gamma_conv}	\\ {\footnotesize $\Gamma$-convergence of	the perturbed fractional Allen-Cahn energy functional}]{Proof of Theorem~\ref{th::gamma_conv} - $\Gamma$-convergence of	the perturbed fractional Allen-Cahn energy functional}
	\label{sec::proof_gamma_conv}
	
	Here, we focus on proving Theorem~\ref{th::gamma_conv}. 
	For this, we point out that a straightforward consequence of Proposition~\ref{prop::L2_cont_conv} is the following statement:
	
	\begin{cor} \label{cor::H_conv}
		Let~$\{u_\epsilon\}_\epsilon$ be a sequence that converge to some function~$u$ in~$L^1_{\loc}(\R^n)$.
		Suppose that~$\{H_\epsilon\}_\epsilon\subset L^\infty(\R^n)$ is such that~$\sup_\epsilon\norm{H_\epsilon}_{L^\infty(\R^n)}<+\infty$, and~$H_\epsilon\to H$ in~$L^1(\Omega)$ as~$\epsilon\to0$.
		
		Then, 
		\begin{equation*}
			\lim_{\epsilon\to0^+} \H_\epsilon(u_\epsilon,\Omega) = \H(u,\Omega).
		\end{equation*}
		
		In addition, the sequence~$\{\H_\epsilon\}_\epsilon$ is pointwise convergent in~$L^1_{\loc}(\R^n)$ to~$\H$.
	\end{cor}
	
	We also recall that~$\F_\epsilon\xrightarrow{\Gamma}\F$, thanks to~\cite[Theorem~1.2]{savin_valdinoci_gamma_conv} 
	As a result, the proof of Theorem~\ref{th::gamma_conv}
	follows from this fact, Corollary~\ref{cor::H_conv} and~\cite[Proposition~6.20]{dal_maso}.
	
	
	\section{Preliminary results for limit stability of minimizers for the perturbed fractional Allen-Cahn energy functional} \label{sec::limit_stability_preliminaries} 

In this section we present two preliminary results useful in proving limit stability of minimizers for the perturbed fractional Allen-Cahn energy functional.

We first establish a convergence result for sequences in~$X_M$ with uniformly bounded energy, as claimed in~\eqref{eq::u_conv}.
\begin{lemma} \label{lemma::u_conv}
	Let~$\{u_\epsilon\}_\epsilon$ be a sequence in~$X_M$.
	If~$\E_\epsilon(u_\epsilon,\Omega)$ is uniformly bounded for a sequence of~$\epsilon\to0^+$, then there exist a convergent subsequence and a set~$E\subseteq\R^n$ such that
	\begin{equation} \label{eq::u_conv2}
		u_\epsilon\to u:=\chi_E-\chi_{E^c} \quad\text{in }L^1(\Omega).
	\end{equation}
\end{lemma}

\begin{proof}
	We observe that, since~$H_\epsilon$ converges to~$H$ in~$L^1(\Omega)$ by~\eqref{eq::H_conv_assumption}, if~$\epsilon$ is sufficiently small,
	$$ \int_\Omega |H_\epsilon|dx\le \int_\Omega |H_\epsilon-H|dx+\int_\Omega |H|dx\le |\Omega|(1+\norm{H}_{L^\infty(\Omega)}).$$
	{F}rom this and the fact that~$\E_\epsilon(u_\epsilon,\Omega)$ is bounded uniformly in~$\epsilon$, we deduce that
	\begin{equation*}
		\begin{split}
			&	|\F_\epsilon(u_\epsilon,\Omega)|\le C + |\H_\epsilon(u_\epsilon,\Omega)|\le C+\frac12\int_\Omega |H_\epsilon|\,|u_\epsilon|dx\\
			&\qquad \leq C+\frac{M}2\int_\Omega |H_\epsilon| dx \le C+\frac{M}{2}|\Omega|(1+\norm{H}_{L^\infty(\Omega)}),
		\end{split}
	\end{equation*}
	for some~$C>0$.
	
	As a result, $\F_\epsilon(u_\epsilon,\Omega)$ is uniformly bounded with respect to the parameter~$\epsilon$. Thus, the convergence in~\eqref{eq::u_conv2} follows from~\cite[Lemma~10]{MR3081641} when~$s\in(0,\frac{1}{2})$, and from~\cite[Proposition~3.3]{savin_valdinoci_gamma_conv} when~$s\in[\frac{1}{2},1)$ (see also Proposition~\ref{prop::bound_conv} in the Appendix).
\end{proof}

Then, we show that minimizers of the
energy functional~$\F_\epsilon$ are uniformly bounded.

\begin{lemma} \label{lemma::unif_bound_min}
	Let~$u\in X_M$ be a minimizer for~$\E_\epsilon$ in~$\Omega$. Then,
	$$\norm{u}_{L^\infty(\R^n)}\leq M.$$
\end{lemma}

\begin{proof}
	Since~$u\in X_M$, then~$|u|\leq M$ in~$\Omega$. 
	
	We now check that~$|u|\le M$ in~$\R^n\setminus\Omega$. For this, suppose by contradiction that~$|u|>M$ in a region with positive measure outside~$\Omega$. Consider the function
	\begin{equation*}
		v:=
		\begin{cases}
			u,&\mbox{if }|u|\leq M,\\
			M,&\mbox{if }u> M,\\
			-M,&\mbox{if }u<- M.
		\end{cases}
	\end{equation*}
	In particular, $v=u$ in~$\Omega$. Moreover,
	\begin{equation*}
		\begin{split}
			&\int_{\Omega}\int_{\Omega^c} \frac{|v(x)-v(y)|}{|x-y|^{n+2s}}dxdy \\
			&\qquad= \int_{\Omega}\int_{\Omega^c\cap\{|u|\leq M\}} \frac{|v(x)-v(y)|}{|x-y|^{n+2s}}dxdy + \int_{\Omega}\int_{\Omega^c\cap \{|u|>M\}} \frac{|v(x)-v(y)|}{|x-y|^{n+2s}}dxdy\\
			&\qquad = \int_{\Omega}\int_{\Omega^c\cap\{|u|\leq M\}} \frac{|u(x)-u(y)|}{|x-y|^{n+2s}}dxdy + \int_{\Omega}\int_{\Omega^c\cap \{u>M\}} \frac{M-u(y)}{|x-y|^{n+2s}}dxdy\\&\qquad\qquad\qquad
			+\int_{\Omega}\int_{\Omega^c\cap \{u<-M\}} \frac{M+u(y)}{|x-y|^{n+2s}}dxdy\\
			&\qquad \leq  \int_{\Omega}\int_{\Omega^c\cap\{|u|\leq M\}} \frac{|u(x)-u(y)|}{|x-y|^{n+2s}}dxdy + \int_{\Omega}\int_{\Omega^c\cap \{|u|>M\}} \frac{|u(x)-u(y)|}{|x-y|^{n+2s}}dxdy\\
			&\qquad = \int_{\Omega}\int_{\Omega^c} \frac{|u(x)-u(y)|}{|x-y|^{n+2s}}dxdy.
		\end{split}
	\end{equation*}
	Hence, 
	\begin{equation*}
		\begin{split}
			&\K(v,\Omega) = \frac{1}{2} \int_{\Omega}\int_{\Omega} \frac{|v(x)-v(y)|}{|x-y|^{n+2s}}dxdy + \int_{\Omega}\int_{\Omega^c} \frac{|v(x)-v(y)|}{|x-y|^{n+2s}}dxdy\\
			&\qquad\leq \frac{1}{2} \int_{\Omega}\int_{\Omega} \frac{|u(x)-u(y)|}{|x-y|^{n+2s}}dxdy + \int_{\Omega}\int_{\Omega^c} \frac{|u(x)-u(y)|}{|x-y|^{n+2s}}dxdy = \K(u,\Omega).
		\end{split}
	\end{equation*}
	
	Therefore, we deduce that
	\begin{equation*}
		\E_\epsilon(v,\Omega) = \K(v,\Omega) + \int_\Omega (W(v) + H\,v)dx \leq \K(u,\Omega) + \int_\Omega (W(u) + H\,u)dx = \E_\epsilon(u,\Omega),
	\end{equation*}
	against the minimality of~$u$. Hence, $|u|\leq M$ in the whole~$\R^n$, as desired.
\end{proof}


\section[Proof of Theorem~\ref{th::min_conv}-(i) when~$s\in(0,1/2)$]{Proof of Theorem~\ref{th::min_conv}-(i) when~$s\in\left(0,\frac{1}{2}\right)$} \label{sec::s_less_half}

In this section, we prove Theorem~\ref{th::min_conv} when~$s\in(0,1/2)$. This is in some sense the easiest configuration, while the case~$s\in[1/2,1)$ will require a bit more effort.

The argument in our proof below is inspired by the proof of~\cite[Theorem~1.3]{savin_valdinoci_gamma_conv} when~$s\in(0,1/2)$.

\begin{proof}[Proof of Theorem~\ref{th::min_conv}-(i)]
	Let~$v\in L^1_{\loc}(\R^n)$ be a function such that~$v=\chi_F-\chi_{F^c}$ in~$\Omega$,
	for some set~$F$, and~$v=u_0$ in~$\Omega^c$.
	
	Let also
	\begin{equation*}
		u_\star:=
		\begin{cases}
			u &\text{ in }\Omega, \\
			u_0 &\text{ in }\Omega^c ,
		\end{cases}
	\end{equation*} where~$u$ is given by Lemma~\ref{lemma::u_conv}. 
	
	We define, for every~$y\in\Omega^c$,
	\begin{equation*}
		\psi(y) := \int_\Omega \frac{v(x)}{|x-y|^{n+2s}} dx \qquad\text{and}\qquad \Psi(y) := \int_\Omega \frac{u_\star(x)}{|x-y|^{n+2s}} dx.
	\end{equation*}
	Since 
	\begin{equation*} 
		\int_{\Omega^c}\int_\Omega \frac{dxdy}{|x-y|^{n+2s}} <+\infty,\quad\mbox{for any }s\in\left(0,\frac{1}{2}\right),
	\end{equation*}
	we have that~$\psi$, $\Psi\in L^1(\Omega^c)$.
	
	Furthermore, by Lemmata~\ref{lemma::u_conv} and~\ref{lemma::unif_bound_min}, we know that~$u_\epsilon$ converges to~$u$ in~$L^1(\Omega)$, and 
	$$\norm{u}_{L^\infty(\R^n)} + \sup_\epsilon\norm{u_\epsilon}_{L^\infty(\R^n)}\leq 1+M.$$
	Thus, using the Dominated Convergence Theorem, we obtain that
	\begin{equation} \label{eq::th_min_conv2}
		\lim_{\epsilon\to0^+} \int_\Omega\int_{\Omega^c} \frac{|u_\epsilon(x)|^2}{|x-y|^{n+2s}}dxdy =  \int_\Omega\int_{\Omega^c} \frac{|u(x)|^2}{|x-y|^{n+2s}}dxdy 
	\end{equation}
	and
	\begin{equation} \label{eq::th_min_conv3}
		\begin{split}
			&\lim_{\epsilon\to0^+} \left| \int_\Omega\int_{\Omega^c} \big(u_\epsilon(x)-u(x)\big)\frac{u_\epsilon(y)}{|x-y|^{n+2s}}dxdy \right| \\
			&\qquad\qquad \leq M
			\lim_{\epsilon\to0^+} \int_\Omega\int_{\Omega^c} \frac{|u_\epsilon(x)-u(x)|}{|x-y|^{n+2s}}dxdy =  0.
		\end{split}
	\end{equation}
	
	In addition, thanks to Fatou's Lemma, 
	\begin{equation} \label{eq::th_min_conv4}
		\int_\Omega \int_\Omega \frac{|u(x)-u(y)|}{|x-y|^{n+2s}} dxdy \leq \liminf_{\epsilon\to0^+} \int_\Omega \int_\Omega \frac{|u_\epsilon(x)-u_\epsilon(y)|}{|x-y|^{n+2s}} dxdy .
	\end{equation}
	Moreover, by definition of weak* convergence in~$L^\infty(\Omega^c)$, we have that
	\begin{equation} \label{eq::th_min_conv1}
		\lim_{\epsilon\to0^+} \int_{\Omega^c} (u_\epsilon(y)-u_0(y))\phi(y)\, dy = 0,\quad\text{ for every }\phi\in L^1(\Omega^c).
	\end{equation}
	
	Now, we consider the function
	\begin{equation*}
		v_\epsilon:=
		\begin{cases}
			v &\text{ in }\Omega, \\
			u_\epsilon &\text{ in }\Omega^c .
		\end{cases}
	\end{equation*}
	Then, we have
	\begin{equation*}
		\begin{split}
			&\E_\epsilon(v_\epsilon,\Omega)-\E_\epsilon(u_\epsilon,\Omega) \\
			=\;&\, \K(v_\epsilon,\Omega)-\K(u_\epsilon,\Omega)-\epsilon^{-2s}\int_\Omega W(u_\epsilon(x))dx +\frac12
			\int_\Omega (v(x)-u_\epsilon(x))H_\epsilon(x) dx \\
			\leq\;& \frac{1}{2}\big(v(\Omega,\Omega)-u_\epsilon(\Omega,\Omega)\big)+\int_\Omega\int_{\Omega^c} \frac{|v(x)-u_\epsilon(y)|^2-|u_\epsilon(x)-u_\epsilon(y)|^2}{|x-y|^{n+2s}}dxdy\\
			&\qquad +\frac12\int_\Omega (v(x)-u_\epsilon(x))H_\epsilon(x) dx \\
			=\;& \frac{1}{2}\big(v(\Omega,\Omega)-u_\epsilon(\Omega,\Omega)\big)
			+\int_\Omega\int_{\Omega^c} \frac{|v(x)|^2-|u_\epsilon(x)|^2}{|x-y|^{n+2s}}dxdy \\
			& \qquad+2\int_{\Omega^c} \big(\Psi(y)-\psi(y)\big) u_\epsilon(y)dy +2\int_\Omega\int_{\Omega^c} \big(u_\epsilon(x)-u(x)\big)\frac{u_\epsilon(y)}{|x-y|^{n+2s}}dxdy\\&\qquad+\frac12\int_\Omega \left(v(x)-u_\epsilon(x)\right)H_\epsilon(x) dx .
		\end{split}
	\end{equation*}
	Therefore, it follows from the minimality of~$u_\epsilon$ that
	\begin{equation*}
		\begin{split}
			0 
			\leq\; &\frac{1}{2}\big(v(\Omega,\Omega)-u_\epsilon(\Omega,\Omega)\big)
			+\int_\Omega\int_{\Omega^c} \frac{|v(x)|^2-|u_\epsilon(x)|^2}{|x-y|^{n+2s}}dxdy \\
			& \qquad+2\int_{\Omega^c} \big(\Psi(y)-\psi(y)\big) u_\epsilon(y)dy +2\int_\Omega\int_{\Omega^c} \big(u_\epsilon(x)-u(x)\big)\frac{u_\epsilon(y)}{|x-y|^{n+2s}}dxdy\\&\qquad+\frac12\int_\Omega \left(v(x)-u_\epsilon(x)\right)H_\epsilon(x) dx.
		\end{split}
	\end{equation*}
	Taking the~$\liminf$ as~${\epsilon\to0^+}$ in the last expression, and using~\eqref{eq::th_min_conv2}, \eqref{eq::th_min_conv3},
	\eqref{eq::th_min_conv4} and~\eqref{eq::th_min_conv1}, we conclude that
	\begin{equation*}
		\begin{split}
			0
			\leq&\, \frac{1}{2}\left(v(\Omega,\Omega)-u(\Omega,\Omega)\right)+\int_\Omega\int_{\Omega^c} \frac{|v(x)|^2-|u(x)|^2}{|x-y|^{n+2s}}dxdy \\
			& +2\int_{\Omega^c} \big(\Psi(y)-\psi(y)\big) u_0(y)dy
			+\frac12\int_\Omega (v(x)-u(x))H(x) dx\\
			=&\, \frac{1}{2}\left(v(\Omega,\Omega)-u_\star(\Omega,\Omega)\right)+\int_\Omega\int_{\Omega^c} \frac{|v(x)|^2-|u_\star(x)|^2}{|x-y|^{n+2s}}dxdy \\
			& +2\int_{\Omega^c} \big(\Psi(y)-\psi(y)\big) u_\star(y)dy
			+\frac12\int_\Omega (v(x)-u_\star(x))H(x) dx\\
			=&\, \E(v,\Omega) - \E(u_\star,\Omega).
		\end{split}
	\end{equation*}
	Namely, $u_\star$ is a minimizer for the functional~$\E$ in~$\Omega$, as desired.
\end{proof}


\section[Proof of Theorem~\ref{th::min_conv}-(ii) when~$s\in[1/2,1)$]{Proof of Theorem~\ref{th::min_conv}-(ii) when~$s\in\left[\frac{1}{2},1\right)$} \label{sec::s_geq_half}

Now, we tackle part~\eqref{enum::min_conv2} of Theorem~\ref{th::min_conv}. For this, we revisit the proof of~\cite[Theorem~1.3]{savin_valdinoci_gamma_conv} to take into account the extra contribution due to the forcing term in the energy functional.

\begin{proof}[Proof of Theorem~\ref{th::min_conv}-\eqref{enum::min_conv2}]
	Let~$K\comp\Omega$ be a smooth set, and take~$\delta>0$ so small that 
	$$K_\delta:=\{x\in\R^n \text{ s.t. }\mbox{dist}(x,K)<\delta\} \subset\Omega .$$	
	Consider also a set~$F$ such that~$F\setminus K = E\setminus K$. 
	
	Thanks to~\cite[Propositions~4.5 and~4.6]{savin_valdinoci_gamma_conv}, there exists a sequence~$\{w_\epsilon\}_\epsilon\subset L^1(K_\delta)$ such that
	\begin{align}
		\label{eq::w_eps_prop1} &w_\epsilon \to \chi_F-\chi_{F^c} \text{ in }L^1(K_\delta)\\
		\label{eq::w_eps_prop2} \text{and } \quad &\lim_{\epsilon\to0^+}\F_\epsilon(w_\epsilon,K_\delta) = c_\star\Per(F,K_\delta),
	\end{align}
	where the constant~$c_\star$ is as in~\eqref{eq::F}.
	
	We point out that, in~$L^1(K_\delta\setminus K)$, the function~$ u_\epsilon-w_\epsilon$ converges to
	$$ \chi_E-\chi_{E^c} - ( \chi_F-\chi_{F^c}) =0.$$
	Thus, we are in the position of using~\cite[Proposition~4.1]{savin_valdinoci_gamma_conv}. In this way, we obtain that
	there exists a sequence of functions~$v_\epsilon$ satisfying\footnote{The notation~$\mathring{K}$ denotes the interior of the set~$K$, namely~$K\setminus\partial K$.
		
		Also, see in particular the formula below~(4.31) in the proof of~\cite[Proposition~4.1]{savin_valdinoci_gamma_conv} for~\eqref{eq::SV_prop_4_1} here.}
	\begin{equation*}
		v_\epsilon =
		\begin{cases}
			w_\epsilon \quad&\text{in }\mathring{K},\\
			u_\epsilon \quad&\text{in } K_\delta^c,
		\end{cases}
	\end{equation*} and
	such that
	\begin{equation} \label{eq::SV_prop_4_1}
		\F_\epsilon(v_\epsilon,\Omega) \leq \F_\epsilon(u_\epsilon,\Omega)- \F_\epsilon(u_\epsilon,K)+ \F_\epsilon(w_\epsilon,K_\delta)+o(\epsilon^{2s-1}).
	\end{equation}
	
	{F}rom this and the definition of~$\E_\epsilon$, we obtain that
	\begin{equation*} 
		\begin{split} 
			& \E_\epsilon(v_\epsilon,\Omega) = \F_\epsilon(v_\epsilon,\Omega)+\H_\epsilon(v_\epsilon,\Omega)\\
			&\leq  \F_\epsilon(u_\epsilon,\Omega)- \F_\epsilon(u_\epsilon,K)+ \F_\epsilon(w_\epsilon,K_\delta) + \H_\epsilon(v_\epsilon,\Omega)+o(\epsilon^{2s-1})\\
			&= \E_\epsilon(u_\epsilon,\Omega)- \E_\epsilon(u_\epsilon,K)+ \F_\epsilon(w_\epsilon,K_\delta)
			+ \H_\epsilon(v_\epsilon,\Omega)-\H_\epsilon(u_\epsilon,\Omega)+\H_\epsilon(u_\epsilon,K)+o(\epsilon^{2s-1})
			\\
			&\leq  \E_\epsilon(u_\epsilon,\Omega)- \E_\epsilon(u_\epsilon,K)+ \F_\epsilon(w_\epsilon,K_\delta)+\H_\epsilon(w_\epsilon,K) -\H_\epsilon(u_\epsilon,K_\delta\setminus K)+o(\epsilon^{2s-1})
			.	\end{split}
	\end{equation*}
	
	Now, since~$u_\epsilon$ is a minimizer for~$\E_\epsilon$, we have that 
	\begin{equation*}
		\E_\epsilon(u_\epsilon,\Omega)  \leq \E_\epsilon(v_\epsilon,\Omega) ,
	\end{equation*}
	and hence
	\begin{equation} \label{eq::min_conv_geqhalf_1}
		\E_\epsilon(u_\epsilon,K)\leq \F_\epsilon(w_\epsilon,K_\delta)+\H_\epsilon(w_\epsilon,K)- \H_\epsilon(u_\epsilon,K_\delta\setminus K)+o(\epsilon^{2s-1}).
	\end{equation}
	Thus, taking the~$\limsup$ for~$\epsilon\to0^+$ on both sides of~\eqref{eq::min_conv_geqhalf_1} and using~\eqref{eq::w_eps_prop2} and Proposition~\ref{prop::L2_cont_conv}, we infer that
	\begin{equation*} 
		\begin{split}
			&	\limsup_{\epsilon\to0^+} \E_\epsilon(u_\epsilon,K) 
			\\ \leq\;& c_\star\Per(F,K_\delta) + \frac{1}{2\omega_{n-1}}\int_{K} (\chi_F-\chi_{F^c})H\, dx 
			-\frac{1}{2\omega_{n-1}}\int_{K_\delta\setminus K} (\chi_E-\chi_{E^c})H\, dx\\
			\leq\;&c_\star\Per(F,K_\delta) + \H(\chi_F-\chi_{F^c},K) 
			+\frac{1}{2\omega_{n-1}}\|H\|_{L^\infty(\R^n)}|K_\delta\setminus K|.
		\end{split}
	\end{equation*}
	Therefore, taking the limit as~$\delta\to0$, we conclude that
	\begin{equation} \label{eq::min_conv_geqhalf_3}
		\limsup_{\epsilon\to0^+} \E_\epsilon(u_\epsilon,K) \leq c_\star\mathscr{P}_{\overline{K}}^{H/c_\star}(F) . 
	\end{equation}
	Accordingly, choosing~$F:=E$, we establish~\eqref{eq::energy_upper_bound} for every smooth set~$K\comp\Omega$. 
	
	In general, if~$K$ is not smooth, we obtain the desired result by a standard approximation with smooth sets~$\{U_j\}_j$ such that~$K\subset U_j\comp\Omega$, for every~$j$. The proof of~\eqref{eq::energy_upper_bound} is thereby complete. 
	
	We now check that~$u$ minimizes~$\E$. By~\cite[Proposition~4.5]{savin_valdinoci_gamma_conv} and Proposition~\ref{prop::L2_cont_conv}, we have that
	\begin{equation} \label{eq::min_conv_geqhalf_4}
		\begin{split}
			c_\star\mathscr{P}_K^{H/c_\star}(E) &= c_\star\Per(E,K) + \frac{1}{2\omega_{n-1}}\int_K H(\chi_E-\chi_{E^c})dx \\
			&\leq \limsup_{\epsilon\to0^+}\F_\epsilon(u_\epsilon,K)+\lim_{\epsilon\to0^+}\H_\epsilon(u_\epsilon,K) \\
			&\leq \limsup_{\epsilon\to0^+}\E_\epsilon(u_\epsilon,K).
		\end{split}
	\end{equation}
	Since the functional~$\mathscr{P}$ is continuous with respect to the~$L^1_{\loc}$-convergence of the domain, using~\eqref{eq::min_conv_geqhalf_3} and~\eqref{eq::min_conv_geqhalf_4} with a sequence of smooth sets~$\{U_j\}_j$ such that~$U_j\to\Omega$, we conclude that
	\begin{equation*}
		\mathscr{P}_\Omega^{H/c_\star}(E)\leq \mathscr{P}_\Omega^{H/c_\star}(F) , \quad\text{ for all~$F$ s.t. }F\setminus\Omega=E\setminus\Omega , 
	\end{equation*} as desired.
\end{proof}


\section[Proof of Theorem~\ref{th::gamma_conv_mass_prescribed}-(ii) when~$s\in[1/2,1)$]{Proof of Theorem~\ref{th::gamma_conv_mass_prescribed}\,-\eqref{item::limsup_mass_prescribed} when~$s\in\left[\frac{1}{2},1\right)$} \label{sec::mass_prescribed_AC}

In this section, we prove Theorem~\ref{th::gamma_conv_mass_prescribed}-\eqref{item::limsup_mass_prescribed} when~$s\in[1/2,1)$.

To do so, we suppose that~$u=\chi_E-\chi_{E^c}$ in~$\Omega$, for some set~$E$, otherwise~$\F(u,\Omega)=+\infty$ and there is nothing to prove. The idea is to consider a sequence~$\{v_\epsilon\}_\epsilon$ such that
\begin{eqnarray}
	\label{item::prop_u_epslion1}&& {\mbox{$v_\epsilon\to u$ in $L^1_{\loc}(\R^n)$, as $\epsilon\to0$}}
	\\ \label{item::prop_u_epslion2} {\mbox{and }}  && \F(u,\Omega) \geq \limsup_{\epsilon\to0^+} \F_\epsilon(v_\epsilon,\Omega)
\end{eqnarray}
and suitably modify~$v_\epsilon$ in such a way that we obtain a sequence~$\{u_\epsilon\}_\epsilon\subset Z_{M,m}$ with the same properties~\eqref{item::prop_u_epslion1} and~\eqref{item::prop_u_epslion2}.

To employ this strategy we recall the following result, needed to construct a suitable recovery sequence:

\begin{lemma}[Proposition~4.6, \cite{savin_valdinoci_gamma_conv}] \label{lemmakgeroith9876}
	Given a set~$E$, there exists a sequence~$v_\epsilon$ that
	converges to~$\chi_E-\chi_{E^c}$ in~$L^1(\Omega)$ such that~$\|v_\epsilon\|_{L^\infty(\R^n)}\le1$ and
	$$ \limsup_{\epsilon\to0^+} \F_\epsilon(v_\epsilon,\Omega) \le c_\star\Per(E,\Omega) .$$
\end{lemma}

We will now suitably perturb the sequence provided by Lemma~\ref{lemmakgeroith9876} in order to make it satisfy the mass constraint. This will be accomplished thanks to the following lemma:

\begin{lemma} \label{lemma::mass_prescribed_recovery_seq}
	Let~$\{v_\epsilon\}_\epsilon$ be the recovery sequence given by Lemma~\ref{lemmakgeroith9876}. 
	
	Then, there exists~$\varepsilon_0>0$ such that, for all~$\epsilon\in(0,\epsilon_0)$, there exists a function~$\phi_\epsilon\in \cont^\infty_c(\Omega)$ such that~$\|\phi_\epsilon\|_{L^\infty(\R^n)}\leq C\epsilon^s$, for some positive constant~$C$ independent of~$\epsilon$, and, setting~$u_\epsilon:=v_\epsilon+\phi_\epsilon$, it holds that~$u_\epsilon\in Z_{M,m}$ and
	\begin{equation*}
		\limsup_{\epsilon\to0^+} \F_\epsilon(u_\epsilon,\Omega) \le \limsup_{\epsilon\to0^+} \F_\epsilon(v_\epsilon,\Omega) .
	\end{equation*}
\end{lemma}

\begin{proof}
	Up to a standard approximation argument, we assume that the set~$E$ has smooth boundary. 
	
	Recall that the sequence~$\{v_\epsilon\}_\epsilon$ is produced  in~\cite[Proposition~4.6]{savin_valdinoci_gamma_conv} as follows.
	Let~$d(x)$ be the signed distance from~$\partial E$, i.e.
	\begin{equation*}
		d(x):=
		\begin{cases}
			\mbox{dist}(x,\partial E), & \text{ if }x\in E,\\
			-\mbox{dist}(x,\partial E), &\text{ if }x\in E^c,
		\end{cases}
	\end{equation*}
	and set	
	$$ v_\epsilon(x) := u_0\left(\frac{d(x)}{\epsilon}\right),$$
	where~$u_0:\R\to[-1,1]$ is the unique nontrivial~$1$-dimensional minimizer of the functional  
	\begin{equation} \label{eq::1D_functional}
		\J_1(u,\Omega) = \K(u,\Omega) + \int_\Omega W(u)dx
	\end{equation}
	(see~\cite[Theorem~2]{MR3081641} and~\cite[Theorem~4.2]{savin_valdinoci_gamma_conv}).
	
	Moreover, from~\cite[Theorem~2]{MR3081641}, we have that, for every~$t\in\R$,
	\begin{equation} \label{eq::v_epsilon_growth}
		|\mbox{sgn}(t)-u_0(t)|\leq \frac{C}{(1+|t|)^{2s}} \qquad\text{and}\qquad |u_0'(t)|\leq \frac{C}{(1+|t|)^{1+2s}}.
	\end{equation}
	
	Now we take~$\alpha>0$, to be assumed conveniently small in what follows,
	and we consider a family of balls~$\{B_{\rho_j}(x_j)\}_{j=0}^{+\infty}$ centered at some~$x_j\in\partial E\cap \Omega$ and with~$\rho_j <\alpha$, such that 
	\begin{align*}
		&\partial E\cap \overline{\Omega}\subseteq \bigcup_{j=0}^{+\infty} B_{\rho_j}(x_j),\\
		\text{and}\quad& \Per(E,\overline{\Omega}) + \alpha \geq \omega_{n-1}\sum_{j=0}^{+\infty} \rho_j^{n-1}.
	\end{align*}
	Besides, by compactness, there exists a finite covering~$\{B_{\rho_j}(x_j)\}_{j=0}^N$ such that
	$$ \partial E\cap \overline{\Omega}\subseteq B:=\bigcup_{j=0}^{N} B_{\rho_j}(x_j). $$
	
	Let us consider a bump-like function~$\phi\in\cont^\infty_c(\Omega\setminus\overline{B})$ such that 
	\begin{equation}\label{eq::media_perturbazione}
		\int_\Omega\phi\,dx=1.
	\end{equation} 
	For every~$\epsilon>0$, we define the function~$u_\epsilon := v_\epsilon + c_\epsilon\phi$, where 
	\begin{equation}\label{eq::media_perturbazione2}
		c_\epsilon := m-\int_\Omega v_\epsilon dx=\int_\Omega u dx-\int_\Omega v_\epsilon dx.
	\end{equation}
	
	We claim that
	\begin{equation}\label{claimru5rgrd7649549}
		|c_\epsilon|\le C\epsilon^s,
	\end{equation}
	for some positive constant~$C$ independent of~$\epsilon$.
	
	To check this, we use the first estimate in~\eqref{eq::v_epsilon_growth} to see that 
	\begin{equation*} 
		\begin{split}
			&|c_\epsilon| \leq \int_\Omega |u-v_\epsilon|dx = \int_{\Omega} \left| \mbox{sgn}(d(x))-u_0\left(\frac{d(x)}{\epsilon}\right)\right|dx\\
			&\qquad\leq C\sqrt{2} \int_{\Omega} \left| \mbox{sgn}(d(x))-u_0\left(\frac{d(x)}{\epsilon}\right)\right|^{1/2}dx \leq C \int_\Omega \left(\frac{\epsilon}{|d(x)|}\right)^{s}dx ,
		\end{split}
	\end{equation*}
	for some constant~$C>0$ depending only on~$n$, $s$ and~$\Omega$, and possibly changing at every step of the computations. 
	
	Then, splitting~$\Omega$ into~$(\partial E)_\eta$ and~$\Omega\setminus (\partial E)_\eta$, where
	$$ (\partial E)_\eta := \{x\in \Omega: d(x)<\eta\},\quad \text{for some fixed }\eta>0, $$
	and using Fermi coordinates, we obtain that
	\begin{equation*}
		|c_\epsilon| \leq C|\Omega|\left(\frac{\epsilon}{\eta}\right)^s + C\int_{(\partial E)_\eta} \left(\frac{\epsilon}{|d(x)|}\right)^{s} dx \leq C\epsilon^s\left(1+\int_{-\eta}^\eta t^{-s}dt\right)  \leq C \epsilon^{s},
	\end{equation*}
	for some positive constant~$C=C(n,s,\Omega,\eta)$ independent of~$\epsilon$. This establishes~\eqref{claimru5rgrd7649549}, as desired.
	
	{F}rom~\eqref{claimru5rgrd7649549}, we infer that, for all~$x\in\Omega$,
	\begin{equation}\label{mnbvcxz09876543qwertyu}
		|u_\epsilon(x)| \le |v_\epsilon(x)|+ |c_\epsilon|\,|\phi(x)|\le 1+C\epsilon^s\|\phi\|_{L^\infty(\R^n)}\le M,
	\end{equation} provided that~$\epsilon$ is sufficiently small.
	
	Moreover, recalling~\eqref{eq::media_perturbazione} and~\eqref{eq::media_perturbazione2},
	\begin{eqnarray*}
		\int_\Omega u_\epsilon dx= \int_\Omega v_\epsilon dx +c_\epsilon\int_\Omega \phi dx =  \int_\Omega v_\epsilon dx +m- \int_\Omega v_\epsilon dx=m.
	\end{eqnarray*}
	This and~\eqref{mnbvcxz09876543qwertyu} entail that~$u_\epsilon\in Z_{M,m}$, provided that~$\varepsilon$ is sufficiently small. 
	
	Now, we set
	\begin{equation*}
		\delta:= \inf_{x\in\Omega\setminus B} |d(x)| >0.
	\end{equation*}
	{F}rom~\eqref{eq::v_epsilon_growth}, we see that, for all~$x\in\Omega\setminus B$,
	$$ |\mbox{sgn}(x)-v_\epsilon(x)|=\left|\mbox{sgn}(x)-u_0\left(\frac{d(x)}{\epsilon}\right)\right|
	\leq \frac{C\epsilon^{2s}}{|d(x)|^{2s}}\le \frac{C\epsilon^{2s}}{\delta^{2s}},$$
	and therefore, recalling also~\eqref{claimru5rgrd7649549},
	$$ |\mbox{sgn}(x)-u_\epsilon(x)|\le  |\mbox{sgn}(x)-v_\epsilon(x)|+C\epsilon^s
	\le
	\frac{C\epsilon^{2s}}{\delta^{2s}}+C\epsilon^s\le C\epsilon^s,$$
	for some positive constant~$C$ independent of~$\epsilon$ and possibly changing from line to line.
	
	As a consequence, recalling the properties of~$W$ in~\eqref{eq::assW}, we find that, for all~$x\in\Omega\setminus B$,
	\begin{eqnarray*}&& W(u_\epsilon(x))= W(u_\epsilon(x))- W(\mbox{sgn}(x))\le C|1-u_\epsilon(x)|^2 |1+u_\epsilon(x)|^2\\&&\qquad\qquad
		\le C|\mbox{sgn}(x)-u_\epsilon(x)|^2
		\le C\epsilon^{2s}.
	\end{eqnarray*}
	Accordingly, up to renaming~$C$,
	\begin{equation} \label{eq::W_to_0}
		\int_{\Omega\setminus B}W(u_\epsilon) dx \le C\epsilon^{2s}.
	\end{equation}
	
	Furthermore, if~$x\in\Omega\setminus B$, then~\eqref{eq::v_epsilon_growth} gives that
	$$ |\nabla v_\epsilon(x)|=\frac1\epsilon\left|u_0'\left(\frac{d(x)}{\epsilon}\right)\right|\le
	\frac{C\epsilon^{2s}}{|d(x)|^{1+2s}}\le \frac{C}{\delta^{1+2s}}.
	$$
	Hence, if~$x\in\Omega\setminus B$ and~$y\in\R^n$ with~$|x-y|\leq\delta/2$, 
	\begin{eqnarray*}
		|u_\epsilon(x)-u_\epsilon(y)| \leq |v_\epsilon(x)-v_\epsilon(y)| + c_\epsilon|\phi(x)-\phi(y)|
		\leq C|x-y|,
	\end{eqnarray*}
	with~$C$ depending also on~$\delta$.
	
	Therefore, for all~$x\in\Omega\setminus B$,
	\begin{equation*}
		\int_{\R^n} \frac{|u_\epsilon(x)-u_\epsilon(y)|^2}{|x-y|^{n+2s}} dy \leq C \left(\int_{0}^{\delta/2} r^{1-2s}dr + \int_{\delta/2}^{+\infty}r^{-1-2s}dr\right) \leq C,
	\end{equation*}
	and hence
	\begin{equation} \label{eq::K_to_0}
		u_\epsilon(\Omega\setminus B,\R^n) \leq C,
	\end{equation}
	up to renaming~$C$ line after line.
	
	Combining the estimates in~\eqref{eq::W_to_0} and~\eqref{eq::K_to_0},
	we obtain that 
	\begin{equation*}
		\limsup_{\epsilon\to0^+} \F_\epsilon(u_\epsilon,\Omega\setminus B)
		= 0,
	\end{equation*}
	and therefore
	$$ \lim_{\epsilon\to0^+} \F_\epsilon(u_\epsilon,\Omega\setminus B)
	= 0.$$
	As a result,	
	\begin{equation} \label{eq::limsup_Omega_B_u}
		\limsup_{\epsilon\to0^+} \F_\epsilon(u_\epsilon,\Omega) = \limsup_{\epsilon\to0^+} \F_\epsilon(u_\epsilon, B).
	\end{equation} 		
	
	Similarly, one can show (see the proof of~\cite[Proposition~4.6]{savin_valdinoci_gamma_conv}) that also
	\begin{equation} \label{eq::limsup_Omega_B_u2} 
		\limsup_{\epsilon\to0^+} \F_\epsilon(v_\epsilon,\Omega) = 
		\limsup_{\epsilon\to0^+} \F_\epsilon(v_\epsilon, B).
	\end{equation}
	
	We also observe that
	\begin{eqnarray*}
		&&\big| u_\epsilon(B,B^c)-v_\epsilon(B,B^c)\big|=\left|\int_B\int_{B^c}\frac{|u_\epsilon(x)-u_\epsilon(y)|^2- |v_\epsilon(x)-v_\epsilon(y)|^2}{|x-y|^{n+2s}} dx dy\right|\\
		&&\qquad=\left|\int_B\int_{\Omega\setminus B}\frac{|v_\epsilon(x)-v_\epsilon(y)-c_\epsilon\phi(y)|^2- |v_\epsilon(x)-v_\epsilon(y)|^2}{|x-y|^{n+2s}} dx dy\right|\\
		&&\qquad =\left|\int_B\int_{\Omega\setminus B}\frac{2(v_\epsilon(x)-v_\epsilon(y))c_\epsilon\phi(y)+ c_\epsilon^2|\phi(y)|^2}{|x-y|^{n+2s}} dx dy\right|\\
		&&\qquad\le\int_B\int_{\Omega\setminus B}\frac{4|c_\epsilon|\,|\phi(y)|+ c_\epsilon^2|\phi(y)|^2}{|x-y|^{n+2s}} dx dy\\&&\qquad\le C\epsilon^s.
	\end{eqnarray*}
	As a consequence, if~$s\in(1/2,1)$,
	\begin{equation} \label{eq::limsup_Omega_B_u3}\begin{split}
			\F_\epsilon(u_\epsilon,B)= \F_\epsilon(v_\epsilon,B) +  \epsilon^{2s-1}\big(u_\epsilon(B,B^c)-v_\epsilon(B,B^c)\big)\le 
			\F_\epsilon(v_\epsilon,B) +  C\epsilon^{3s-1},
	\end{split}\end{equation} which gives that
	$$ \limsup_{\epsilon\to0^+} \F_\epsilon(u_\epsilon,B)\le \limsup_{\epsilon\to0^+} \F_\epsilon(v_\epsilon, B).$$
	A similar computation works for~$s=1/2$.
	
	Thus, from~\eqref{eq::limsup_Omega_B_u}, \eqref{eq::limsup_Omega_B_u2} and~\eqref{eq::limsup_Omega_B_u3},
	we infer that
	\begin{equation*}
		\begin{split}
			&\limsup_{\epsilon\to0^+} \F_\epsilon(u_\epsilon,\Omega) = \limsup_{\epsilon\to0^+} \F_\epsilon(u_\epsilon,B) \\
			&\qquad \le \limsup_{\epsilon\to0^+} \F_\epsilon(v_\epsilon, B) = \limsup_{\epsilon\to0^+} \F_\epsilon(v_\epsilon, \Omega),
		\end{split}
	\end{equation*}
	concluding the proof.
\end{proof}

\begin{proof}[Proof of Theorem~\ref{th::gamma_conv_mass_prescribed}-\eqref{item::limsup_mass_prescribed} when~$s\in[1/2,1)$]
	Let~$u\in Z_{M,m}\subset X_M$
	and let~$\{u_\epsilon\}_\epsilon\subset Z_{M,m}$
	be the recovery sequence constructed in Lemma~\ref{lemma::mass_prescribed_recovery_seq}.
	Then, Lemmata~\ref{lemmakgeroith9876} and~\ref{lemma::mass_prescribed_recovery_seq} entail that
	\begin{equation*}
		\F(u,\Omega) \geq \limsup_{\epsilon\to0^+} \F_\epsilon(v_\epsilon,\Omega) \ge \limsup_{\epsilon\to0^+} \F_\epsilon(u_\epsilon,\Omega),
	\end{equation*}
	as desired.
\end{proof}


\section[Proof of Theorem~\ref{th::energy_uniform_estimate}	\\ {\footnotesize Uniform energy bound for minimizers with a mass constraint}]{Proof of Theorem~\ref{th::energy_uniform_estimate} - {Uniform energy bound for minimizers with a mass constraint}}
\label{sec::energy_uniform_estimate} 

In this section, we prove that a sequence of solutions for~\eqref{eq::prescribed_mass_AC} has uniformly bounded energy as stated in Theorem~\ref{th::energy_uniform_estimate}. This result is the analog in our setting of~\cite[Theorem~1.3]{savin_valdinoci_density_estimates}.

\begin{proof}[Proof of Theorem~\ref{th::energy_uniform_estimate}]
	Let~$E\subseteq\R^n$ be such that~$|\Omega\cap E |=\frac{m+|\Omega|}{2}$
	and let~$v:=\chi_{{E}}-\chi_{{E}^c}$. Clearly~$\|v\|_{L^\infty(\R^n)}\leq 1 < M$. Moreover, notice that
	$$ \int_\Omega v\,dx=|\Omega\cap  E|-|\Omega\setminus E|=2|\Omega\cap  E|-|\Omega| =m+|\Omega|-|\Omega|=m.	$$
	That is, $v$ belongs to~$Z_{M,m}$.
	
	Then, we can exploit Theorem~\ref{th::gamma_conv_mass_prescribed} and obtain that there exists a sequence~$\{v_\epsilon\}_\epsilon\subset Z_{M,m}$ such that~$v_\epsilon\to v$ in~$L^1_{\loc}(\R^n)$, and
	\begin{equation*} \label{eq::limsup_unif_bound}
		\limsup_{\epsilon\to0^+}\F_\epsilon(v_\epsilon,\Omega)\leq \F(v,\Omega).
	\end{equation*}
	Hence,  the minimality of~$u_\epsilon$ entails that 
	\begin{equation*}
		\limsup_{\epsilon\to0^+} \F_\epsilon(u_\epsilon,\Omega) \leq\F(v,\Omega),
	\end{equation*}
	from which we infer that~$\F_\epsilon(u_\epsilon,\Omega)$ is bounded uniformly in~$\epsilon$, as desired. 
\end{proof}


\section[Proof of Theorem~\ref{th::u_def_outside_omega} \\ {\footnotesize External behavior of solutions of the prescribed-mass non-local Allen-Cahn Equation}]{Proof of Theorem~\ref{th::u_def_outside_omega} - External behavior of solutions of the prescribed-mass non-local Allen-Cahn Equation}
\label{sec::u_def_outside_omega}

In this section, we show that a solution~$u_\epsilon$ of~\eqref{eq::prescribed_mass_AC} satisfies a non-local Neumann external condition. Such condition entails a tight link between the behavior of~$u_\epsilon$ outside and inside~$\Omega$, as stated in Theorem~\ref{th::u_def_outside_omega}. For more details and results related to the non-local Neumann external condition, we refer to~\cite{MR3651008}.

\begin{proof}[Proof of Theorem~\ref{th::u_def_outside_omega}]
	Let~$u_\epsilon\in Z_{M,m}$ be a minimizer for~$\F_\epsilon$ in~$\Omega$. Then, by the minimality of~$u_\epsilon$, for every~$\phi\in\cont^\infty_c(\Omega^c)$, we have that~$u_\epsilon+\phi \in Z_{M,m}$, and
	\begin{equation} \label{eq::u_epsilon_min}
		\F_\epsilon(u_\epsilon,\Omega)\leq \F_\epsilon(u_\epsilon+\phi,\Omega).
	\end{equation}
	Since~${\rm{supp}}(\phi)\comp\Omega^c$, the inequality in~\eqref{eq::u_epsilon_min} entails that
	\begin{eqnarray*}
		0&\le & 	\K(u_\epsilon+\phi,\Omega)-\K(u_\epsilon,\Omega)\\& =&
		\int_\Omega\int_{\Omega^c}\frac{|(u_\epsilon+\phi)(x)- (u_\epsilon+\phi)(y)|^2-|u_\epsilon(x)-u_\epsilon(y)|^2}{|x-y|^{n+2s}}dxdy\\
		&=&\int_\Omega\int_{\Omega^c}\frac{2(u_\epsilon(x)-u_\epsilon(y))(\phi(x)-\phi(y))+|\phi(x)-\phi(y)|^2}{|x-y|^{n+2s}}dxdy\\&=&\int_\Omega\int_{\Omega^c}\frac{-2(u_\epsilon(x)-u_\epsilon(y))\phi(y)+|\phi(y)|^2}{|x-y|^{n+2s}}dxdy,
	\end{eqnarray*}
	namely
	\begin{equation}  \label{eq::diff_condition}
		2\int_\Omega\int_{\Omega^c} \frac{(u_\epsilon(x)-u_\epsilon(y))\phi(y)}{|x-y|^{n+2s}}dxdy \leq \int_\Omega\int_{\Omega^c} \frac{|\phi(y)|^2}{|x-y|^{n+2s}}dxdy=:c_\phi.
	\end{equation}
	
	Notice that~$c_\phi=0$ if and only if~$\phi\equiv0$. Therefore, if~$\phi$ does not vanish identically, we define~$\psi:=\phi/c_\phi$ and we use~$\psi$ as
	a test function in~\eqref{eq::diff_condition}. In this way, we obtain that
	\begin{equation} \label{eq::diff_condition2}
		2\int_\Omega\int_{\Omega^c} \frac{(u_\epsilon(x)-u_\epsilon(y))\phi(y)}{|x-y|^{n+2s}}dxdy \leq 1.
	\end{equation}
	
	Now, if 	\begin{equation*}
		\int_\Omega\int_{\Omega^c} \frac{(u_\epsilon(x)-u_\epsilon(y))\phi(y)}{|x-y|^{n+2s}}dxdy \neq 0,
	\end{equation*}
	we define
	\begin{equation*}
		\widetilde{\phi}(x):=\frac{\phi(x)}{\displaystyle\int_\Omega\int_{\Omega^c} \frac{(u_\epsilon(x)-u_\epsilon(y))\phi(y)}{|x-y|^{n+2s}}dxdy}.
	\end{equation*}
	Using~$\widetilde\phi$ as a test function in~\eqref{eq::diff_condition2}, we thus see that
	\begin{eqnarray*}
		2=2\int_\Omega\int_{\Omega^c} \frac{(u_\epsilon(x)-u_\epsilon(y))\widetilde\phi(y)}{|x-y|^{n+2s}}dxdy \leq 1,
	\end{eqnarray*}	
	which is a contradiction. 
	
	Therefore, for all~$\phi\in\cont^\infty_c(\Omega^c)$ with~$\phi\not\equiv0$, it must be that
	\begin{equation*}
		\int_\Omega\int_{\Omega^c} \frac{(u_\epsilon(x)-u_\epsilon(y))\phi(y)}{|x-y|^{n+2s}}dxdy = 0.
	\end{equation*}
	{F}rom this we infer that
	\begin{equation*}
		\int_\Omega \frac{u_\epsilon(x)-u_\epsilon(y)}{|x-y|^{n+2s}}dy = 0,\quad{\mbox{ for a.e. }} x\in\Omega^c
	\end{equation*} and the desired result is established.
\end{proof}


\section[Proof of Theorem~\ref{th::multiplier_bound}\\ {\footnotesize Convergence of the Lagrange multipliers for the fractional Allen-Cahn energy functional under a mass constraint}]{Proof of Theorem~\ref{th::multiplier_bound} - Convergence of the Lagrange multipliers for the fractional Allen-Cahn energy functional under a mass constraint}
\label{sec::lagr_mult_conv_prescribed_mass}

This section is devoted to the Proof of Theorem~\ref{th::multiplier_bound}.

\begin{proof}[Proof of Theorem~\ref{th::multiplier_bound}]
	Since~$u_\epsilon$ is a solution of~\eqref{eq::inhomog_mass_AC} in~$\Omega$,
	applying iteratively~\cite[Propositions~2.8 and~2.9]{silvestre_regularity} we deduce that~$u_\epsilon$ is smooth in~$\Omega$. 
	
	Therefore, for every vector-valued function~$\psi\in\cont^\infty_c(\Omega,\R^n)$, we have that
	\begin{eqnarray*}0&=&
		\int_{\Omega} \big(\kappa_\epsilon\epsilon^{2s}(-\Delta)^su_\epsilon + \kappa_\epsilon W'(u_\epsilon) + \mu_\epsilon \big)(\psi\cdot \nabla u_\epsilon)dx \\
		&=&\int_{\Omega} \kappa_\epsilon\epsilon^{2s}(-\Delta)^su_\epsilon(\psi\cdot \nabla u_\epsilon)
		+ \kappa_\epsilon\nabla (W(u_\epsilon))\cdot\psi + \mu_\epsilon (\psi\cdot \nabla u_\epsilon)dx .
	\end{eqnarray*}
	Thus, after an integration by parts, we find that
	\begin{equation*}
		\int_{\Omega} \kappa_\epsilon\epsilon^{2s}(-\Delta)^su_\epsilon (\psi\cdot \nabla u_\epsilon) + \big(\kappa_\epsilon W(u_\epsilon) + \mu_\epsilon u_\epsilon \big)\mbox{div}\psi\, dx = 0,
	\end{equation*}
	and hence
	\begin{equation} \label{eq::var_eq}
		|\mu_\epsilon|\left| \int_{\Omega} u_\epsilon\,\mbox{div}\psi\, dx \right| = \left| \int_{\Omega} \kappa_\epsilon\epsilon^{2s}(-\Delta)^su_\epsilon (\psi\cdot \nabla u_\epsilon)\, dx+ \int_{\Omega} \kappa_\epsilon W(u_\epsilon) \,\mbox{div}\psi\, dx\right|. 
	\end{equation}
	
	Now we estimate the first term on the right-hand side of~\eqref{eq::var_eq}. By the definition of the fractional
	Laplacian, using also that~$\mbox{supp}(\psi)\comp \Omega$, we obtain that
	\begin{equation*}
		\begin{split}
			\int_{\Omega} (-\Delta)^su_\epsilon (\psi\cdot \nabla u_\epsilon)\, dx
			&= 2\int_{\R^n}\left(\int_{\R^n} \frac{u_\epsilon(x)-u_\epsilon(y)}{|x-y|^{n+2s}}dy\right) \psi(x)\cdot \nabla_xu_\epsilon(x)\, dx\\
			&= 2\int_{\R^n}\left(\int_{\R^n} \frac{u_\epsilon(x)-u_\epsilon(y)}{|x-y|^{n+2s}}dy\right) \psi(x)\cdot \nabla_x(u_\epsilon(x)-u_\epsilon(y))\, dx\\
			&= \int_{\R^n} \int_{\R^n} \frac{\nabla_x\big((u_\epsilon(x)-u_\epsilon(y))^2\big)}{|x-y|^{n+2s}}\cdot \psi(x)\, dxdy	.
		\end{split}
	\end{equation*}
	Thus, after an integration by parts, we are led to
	\begin{equation}\label{ygkw683498eyvcmhsdvfkhedshiut4yorDDFgj}
		\begin{split}
			&\int_{\Omega} (-\Delta)^su_\epsilon( \psi\cdot \nabla u_\epsilon)\, dx\\ 
			=& -\int_{\R^n}\int_{\R^n} \frac{|u_\epsilon(x)-u_\epsilon(y)|^2}{|x-y|^{n+2s}} \left(\mbox{div}_x\psi(x) -(n+2s)\frac{\psi(x)\cdot(x-y)}{|x-y|^2}\right)\, dxdy\\ 
			=& -\int_{\Omega}\int_{\R^n} \frac{|u_\epsilon(x)-u_\epsilon(y)|^2}{|x-y|^{n+2s}} \left(\mbox{div}_x\psi(x) -(n+2s)\frac{\psi(x)\cdot(x-y)}{|x-y|^2}\right)\, dxdy.
		\end{split}
	\end{equation}
	
	We observe that
	\begin{equation*}
		\begin{split}
			&\int_{\Omega}\int_{\Omega} \frac{|u_\epsilon(x)-u_\epsilon(y)|^2\psi(x)\cdot(x-y)}{|x-y|^{n+2s+2}} \, dxdy\\
			&\qquad\qquad\qquad =\frac12\int_{\Omega}\int_{\Omega} \frac{|u_\epsilon(x)-u_\epsilon(y)|^2\psi(x)\cdot(x-y)}{|x-y|^{n+2s+2}} \, dxdy\\
			&\qquad\qquad\qquad\qquad +\frac12\int_{\Omega}\int_{\Omega} \frac{|u_\epsilon(x)-u_\epsilon(y)|^2\psi(x)\cdot(x-y)}{|x-y|^{n+2s+2}} \, dxdy\\
			&\qquad\qquad\qquad =\frac12\int_{\Omega}\int_{\Omega} \frac{|u_\epsilon(x)-u_\epsilon(y)|^2\psi(x)\cdot(x-y)}{|x-y|^{n+2s+2}} \, dxdy\\
			&\qquad\qquad\qquad\qquad+\frac12\int_{\Omega}\int_{\Omega} \frac{|u_\epsilon(y)-u_\epsilon(x)|^2\psi(y)\cdot(y-x)}{|y-x|^{n+2s+2}} \, dxdy\\
			&\qquad\qquad\qquad =\frac12\int_{\Omega}\int_{\Omega} \frac{|u_\epsilon(x)-u_\epsilon(y)|^2(\psi(x)-\psi(y))\cdot(x-y)}{|x-y|^{n+2s+2}} \, dxdy.
		\end{split}
	\end{equation*}
	Plugging this information into~\eqref{ygkw683498eyvcmhsdvfkhedshiut4yorDDFgj}, we find that
	\begin{equation*}
		\begin{split}
			&\int_{\Omega} (-\Delta)^su_\epsilon (\psi\cdot \nabla u_\epsilon)\, dx\\ 
			=& -\int_{\Omega}\int_{\Omega} \frac{|u_\epsilon(x)-u_\epsilon(y)|^2}{|x-y|^{n+2s}} \left(\mbox{div}_x\psi(x)- \frac{(n+2s)(\psi(x)-\psi(y))\cdot(x-y)}{2|x-y|^2}\right)\, dxdy\\
			& -\int_{\Omega}\int_{\Omega^c} \frac{|u_\epsilon(x)-u_\epsilon(y)|^2}{|x-y|^{n+2s}} \left(\mbox{div}_x\psi(x)-(n+2s)\frac{\psi(x)\cdot(x-y)}{|x-y|^2}\right)\, dxdy	.
		\end{split}
	\end{equation*}
	As a result,
	\begin{equation*}
		\left|	\int_{\Omega} (-\Delta)^su_\epsilon (\psi\cdot \nabla u_\epsilon)\, dx\right|\le
		C\int_{\Omega}\int_{\R^n} \frac{|u_\epsilon(x)-u_\epsilon(y)|^2}{|x-y|^{n+2s}}\, dxdy,
	\end{equation*} for some~$C>0$, depending on~$n$, $s$ and~$\psi$.
	
	{F}rom this estimate, recalling the definition of~$\kappa_\epsilon$ in~\eqref{kappaepsilondef}, we deduce that
	\begin{equation} \label{eq::laplace_term_estimate}\begin{split}
			&	\kappa_\epsilon\epsilon^{2s}\left|  \int_{\Omega} (-\Delta)^su_\epsilon (\psi\cdot \nabla u_\epsilon)\, dx \right|
			\le C\kappa_\epsilon\epsilon^{2s}\int_{\Omega}\int_{\R^n} \frac{|u_\epsilon(x)-u_\epsilon(y)|^2}{|x-y|^{n+2s}}\, dxdy
			\\&\qquad		= C\kappa_\epsilon\epsilon^{2s}\K(u_\epsilon,\Omega)\le C\kappa_\epsilon\J_\epsilon(u_\epsilon,\Omega)
			= C \F_\epsilon(u_\epsilon,\Omega).
	\end{split}\end{equation}
	
	Moreover, the second term on the right-hand side in~\eqref{eq::var_eq} can be estimated as
	\begin{equation} \label{eq::W_term_estimate}
		\kappa_\epsilon\left|\int_{\Omega} W(u_\epsilon) \mbox{div}\psi\, dx\right| \leq C \F_\epsilon(u_\epsilon,\Omega).
	\end{equation}
	
	Besides, from~\eqref{eq::energy_unif_estimateBIS} we know that~$\F_\epsilon(u_\epsilon,\Omega)$ is bounded uniformly in~$\epsilon$. Hence, from~\eqref{eq::var_eq}, \eqref{eq::laplace_term_estimate} and~\eqref{eq::W_term_estimate}, we find that
	\begin{equation} \label{eq::mu_epsilon_bound2}
		|\mu_\epsilon| \left|\int_{\Omega} u_\epsilon\mbox{div}\psi\, dx\right| \leq  C.
	\end{equation}
	for some positive~$C$ independent of~$\epsilon$.
	
	We now claim that
	\begin{equation}\label{hgvjckxlz7648392hfjdks}
		\Per(E,\Omega)\in(0,+\infty].
	\end{equation} 
	Indeed, if~$\Per(E,\Omega)=0$, we deduce from the relative isoperimetric inequality (see e.g.~\cite[page~87]{MR1736268}) that
	$$ 0=\Per(E,\Omega)\ge C(n,\Omega) |E\cap \Omega|\; |\Omega\setminus E|,$$
	for some~$C(n,\Omega)>0$. 
	
	As a result, either~$|E\cap \Omega|=0$ or~$|\Omega\setminus E|=0$, and therefore
	\begin{eqnarray*}&& |\Omega|=\Big| |E\cap \Omega|-|\Omega\setminus E|\Big|=\left|\int_\Omega \chi_E-\chi_{E^c}dx\right|
		=\lim_{\epsilon\to0}\left|\int_\Omega u_\epsilon dx\right|=|m|,
	\end{eqnarray*} in contradiction with our assumption. This establishes~\eqref{hgvjckxlz7648392hfjdks}.
	
	Accordingly, by~\eqref{hgvjckxlz7648392hfjdks}, and recalling the definition of perimeter (see e.g.~\cite[page~22]{MR1736268}), there exists~$\psi\in\cont^\infty_c(\Omega,\R^n)$
	such that
	\begin{equation}\label{ghvjckxlaO32765U8I43}
		\int_E \mbox{div}\psi\,dx\ge\frac12 \min\{ \Per(E,\Omega), 1\} .\end{equation}
	
	Moreover, 
	\begin{eqnarray*}&&0=
		\int_{\partial\Omega} \psi\cdot\nu_\Omega\, dx= \int_\Omega \mbox{div}\psi\,dx=\int_E \mbox{div}\psi\,dx+\int_{E^c} \mbox{div}\psi\,dx,
	\end{eqnarray*}	and therefore
	$$ \int_{E^c} \mbox{div}\psi\,dx=-\int_E \mbox{div}\psi\,dx.$$
	As a consequence of this and~\eqref{ghvjckxlaO32765U8I43},
	$$ \int_{\Omega} u\, \mbox{div}\psi\, dx=\int_E \mbox{div}\psi\,dx-\int_{E^c} \mbox{div}\psi\,dx=2\int_E \mbox{div}\psi\,dx\ge  \min\{ \Per(E,\Omega), 1\}.$$
	{F}rom this fact, since~$u_\epsilon\to u=\chi_E-\chi_{E^c}$ in~$L^1(\Omega)$, we arrive at
	\begin{equation*} 
		\left|\int_{\Omega} u_\epsilon\mbox{div}\psi\, dx\right|  = \left|\int_{\Omega} u\, \mbox{div}\psi\, dx + o(1)\right| \ge \frac12  \min\{ \Per(E,\Omega), 1\}.
	\end{equation*}
	
	This and~\eqref{eq::mu_epsilon_bound2} entail that
	\begin{equation*}
		|\mu_\epsilon| \leq \frac{2C}{\displaystyle\min\{ \Per(E,\Omega), 1\}}.
	\end{equation*}
	Accordingly, there exists~$\mu\in\R$ such that~$\mu_\epsilon\to\mu$ (up to a subsequence) as~$\epsilon\to0$, and this concludes the proof of Theorem~\ref{th::multiplier_bound}. 
\end{proof}


\section[Proof of Theorem~\ref{th::energy_continuity}\\ {\footnotesize Energy limits under a mass constraint}]{Proof of Theorem~\ref{th::energy_continuity} - Energy limits under a mass constraint}
\label{sec::min_conv_prescribed_mass} 

In this section, we give the proof of Theorem~\ref{th::energy_continuity} concerning the
continuity of the energy functionals~$\G_\epsilon$ with respect to the~$L^1_{\loc}$-convergence of minimizers.

\begin{proof}[Proof of Theorem~\ref{th::energy_continuity}]	
	Thanks to Theorem~\ref{th::energy_uniform_estimate}, there exists~$C$ independent of~$\epsilon$ such that~$\F_\epsilon(u_\epsilon,\Omega) \leq C$,
	hence, by~\cite[Lemma~10]{MR3081641}, there exists~$E$ such that~$u_\epsilon\to \chi_E-\chi_{E^c}$ in~$L^1(\Omega)$ (up to a subsequence).
	
	Furthermore, thanks to Lemma~\ref{lemma::unif_bound_min}, we have that~$\sup_\epsilon\norm{u_\epsilon}_{L^\infty(\R^n)} \leq M$. Hence, we deduce from~\eqref{eq::u_def_outside_omega} and the Dominated Convergence Theorem that
	\begin{equation*}
		\lim_{\epsilon\to0}u_\epsilon(x)=  \frac{\displaystyle\int_\Omega \frac{\chi_E(y)-\chi_{E^c}(y)}{|x-y|^{n+2s}}dy }{\displaystyle\int_\Omega \frac{1}{|x-y|^{n+2s}}dy },\quad \text{ for a.e. } x\in\Omega^c.
	\end{equation*}
	
	Therefore, setting
	\begin{equation*}
		u:=
		\begin{cases}
			\chi_E-\chi_{E^c}\quad&\text{in }\Omega,\\
			\\
			\frac{\displaystyle\int_\Omega \frac{\chi_E(y)-\chi_{E^c}(y)}{|x-y|^{n+2s}}dy }{\displaystyle\int_\Omega \frac{1}{|x-y|^{n+2s}}dy }\quad&\text{in }\Omega^c,
		\end{cases}
	\end{equation*}
	we conclude that~$u_\epsilon\to {u}$ in~$L^1_{\loc}(\R^n)$ and that~$\|u\|_{L^\infty(\R^n)}\le1$. 
	
	Moreover,
	\begin{equation*}	
		\left|m-\int_{\Omega} {u}\,dx\right| 
		=	\lim_{\epsilon\to0} \left|\int_{\Omega} \big(
		u_\epsilon-{u}\big)\,dx\right|
		\leq \lim_{\epsilon\to0}\norm{u_\epsilon-u}_{L^1(\Omega)}=0,
	\end{equation*}
	from which we deduce that~${u}\in Z_{M,m}$.

	Now, by the lower-semicontinuity of~$\F_\epsilon$, we have that
	\begin{equation} \label{eq::liminf_prescribed_mass}
		\F({u},\Omega)\leq\liminf_{\epsilon\to0^+}\F_\epsilon(u_\epsilon,\Omega).
	\end{equation}
	
	Furthermore, as a consequence of Theorem~\ref{th::gamma_conv_mass_prescribed}-\eqref{item::limsup_mass_prescribed}, there exists a sequence~$\{w_\epsilon\}_\epsilon\subset Z_{M,m}$ such that 
	\begin{equation*}
		\F({u},\Omega)\geq\limsup_{\epsilon\to0^+}\F_\epsilon(w_\epsilon,\Omega).
	\end{equation*}
	By the minimality of~$u_\epsilon$, it thereby follows that
	\begin{equation*}
		\F({u},\Omega)\geq\limsup_{\epsilon\to0^+}\F_\epsilon(u_\epsilon,\Omega).
	\end{equation*}
	This and~\eqref{eq::liminf_prescribed_mass} entail that
	\begin{equation} \label{eq::F_continuity_prescribed_mass}
		\lim_{\epsilon\to0^+} \F_\epsilon(u_\epsilon,\Omega) = \F({u},\Omega) ,
	\end{equation} which is~\eqref{eq::F_convmass_prescribed}.
	
	We now check that~$u$ is a minimizer for~$\F$ under the mass constraint.	 For this, let~$v:=\chi_F-\chi_{F^c}\in Z_{M,m}$, for some~$F\subset\R^n$.
	Thanks to Theorem~\ref{th::gamma_conv_mass_prescribed}, there exists a sequence~$\{v_\epsilon\}_\epsilon\subset Z_{M,m}$ such that~$v_\epsilon\to v$ in~$L^1_{\loc}(\R^n)$ and
	\begin{equation*}
		\lim_{\epsilon\to0^+} \F_\epsilon(v_\epsilon,\Omega) = \F(v,\Omega) .
	\end{equation*}
	Therefore, using also the minimality of~$u_\epsilon$ and~\eqref{eq::F_continuity_prescribed_mass}, we infer that
	\begin{equation} \label{eq::u_tilde_min_F}
		\F({u},\Omega) = \lim_{\epsilon\to0^+}\F_\epsilon(u_\epsilon,\Omega) \leq \lim_{\epsilon\to0^+} \F_\epsilon(v_\epsilon,\Omega) = \F(v,\Omega) .
	\end{equation}
	Since~$\F(v,\Omega)=+\infty$ whenever~$v$ is not a signed characteristics function in~$\Omega$, inequality~\eqref{eq::u_tilde_min_F} yields that
	\begin{equation*}
		\F({u},\Omega) = \min_{v\in Z_{M,m}} \F(v,\Omega),
	\end{equation*} as desired.
	
	Finally, thanks to~\eqref{eq::F_convmass_prescribed} and Corollary~\ref{cor::H_conv}, we conclude that 
	\begin{equation*}
		\lim_{\epsilon\to0^+} \G_\epsilon(u_\epsilon,\Omega) = \G({u},\Omega) ,
	\end{equation*}
	which establishes~\eqref{eq::G_convmass_prescribed}, as well.
\end{proof}


\section[Proof of Theorem~\ref{th::millot_sire_wang}\\ {\footnotesize Convergence for solutions of inhomogeneous Allen-Cahn Equation}]{Proof of Theorem~\ref{th::millot_sire_wang} - Convergence for solutions of inhomogeneous Allen-Cahn Equation}
\label{sec::millot_sire_wang}

Here, we prove a version of~\cite[Theorem~5.1]{MR3900821} needed for our purposes, as stated in Theorem~\ref{th::millot_sire_wang}.

\begin{proof}[Proof of Theorem~\ref{th::millot_sire_wang}]
	{F}rom the assumptions on~$g_k$ and~$g$, we have that~$g\in L^\infty(\Omega^c)$. Then, there exists a sequence~$\{\tilde{g}_k\}_k\subseteq\cont^{0,1}_{\loc}(\R^n)$ such that~$\sup_k \|\tilde{g}_k\|_{L^\infty(\Omega^c)}<+\infty$, $\tilde{g}_k\to g$ in~$L^1_{\loc}(\Omega_\eta^c)$, and~$\tilde{g}_k=g_k$ in~$\Omega_\eta$. 
	
	Now, let us define~$\tilde{u}_k\in H^s(\Omega)\cap L^p(\Omega)$ as
	\begin{equation*}
		\tilde{u}_k:=
		\begin{cases}
			u_k\quad &\text{in }\Omega,\\
			\tilde{g}_k\quad &\text{in }\Omega^c,
		\end{cases}
	\end{equation*}
	and observe that~$W'(\tilde{u}_k)=W'(u_k)$ in~$\Omega$. 
	
	Notice also that, for all~$y\in\Omega_\eta\setminus\Omega$,
	\begin{equation*}
		u_k(y)={g}_k(y)=\tilde{g}_k(y)=\tilde{u}_k(y).
	\end{equation*}
	Therefore, omitting the principal value notation for the sake of readability, we see that, for all~$x\in\Omega$,
	\begin{equation*}
		\begin{split}
			(-\Delta)^s \tilde{u}_k(x) &=2 \int_{\R^n} \frac{\tilde{u}_k(x)-\tilde{u}_k(y)}{|x-y|^{n+2s}}dy\\&\
			=  2\int_{\R^n} \frac{u_k(x)-u_k(y)}{|x-y|^{n+2s}}dy + 2\int_{\Omega_\eta^c} \frac{u_k(y)-\tilde{u}_k(y)}{|x-y|^{n+2s}}dy\\
			&= (-\Delta)^su_k(x)+h_k(x),
		\end{split}
	\end{equation*}
	where
	\begin{equation*}
		h_k(x):=2\int_{\Omega_\eta^c} \frac{u_k(y)-\tilde{u}_k(y)}{|x-y|^{n+2s}}dy
		= 2\int_{\Omega_\eta^c} \frac{g_k(y)-\tilde{g}_k(y)}{|x-y|^{n+2s}}dy.
	\end{equation*}
	Notice that~$h_k\in \cont^{0,1}(\Omega)$ with
	\begin{equation}\label{tfghdsjkcnxzm5647567438fghdj}
		\sup_k\| h_k\|_{W^{1,\infty}(\Omega)}<+\infty.\end{equation}
	
	We claim that
	\begin{equation}\label{tfghdsjkcnxzm5647567438fghdj2}
		{\mbox{$h_k\rightharpoonup0$ in~$W^{1,q}(\Omega)$ as~$k\to+\infty$.}}
	\end{equation}
	For this, for~$x\in\Omega$, we observe that the functions~$\R^n\ni y\mapsto|x-y|^{-n-2s}\chi_{\Omega_\eta^c}(y)$ and~$\R^n\ni y\mapsto(x-y)|x-y|^{-n-2s-2}\chi_{\Omega_\eta^c}(y)$ belong to~$L^1(\R^n)$. Thus,
	by the Dominated Convergence Theorem, we infer that, for all~$x\in\Omega$, up to subsequences, 
	\begin{eqnarray*}&&\lim_{k\to+\infty} h_k(x)=
		\lim_{k\to+\infty}
		\int_{\Omega_\eta^c} \frac{ g_k(y)-\tilde{g}_k(y)}{  |x-y|^{n+2s}}dy=0
	\end{eqnarray*}
	and
	\begin{eqnarray*}&&\lim_{k\to+\infty} \nabla h_k(x)=
		-(n+2s)\lim_{k\to+\infty} \int_{\Omega_\eta^c} \frac{ \big(g_k(y)-\tilde{g}_k(y)\big)(x-y)}{  |x-y|^{n+2s+2}}dy=0.
	\end{eqnarray*}
	As a consequence, recalling~\eqref{tfghdsjkcnxzm5647567438fghdj} and making again use of the Dominated Convergence Theorem, we obtain~\eqref{tfghdsjkcnxzm5647567438fghdj2}.
	
	Now we observe that~$\tilde{u}_k$ weakly solves
	\begin{equation*}
		\begin{cases}
			(-\Delta)^sv+\epsilon_k^{-2s}W'(v)=f_k+h_k\quad &\text{in }\Omega,\\
			v=\tilde{g}_k\quad &\text{in }\Omega^c.
		\end{cases}
	\end{equation*}
	{F}rom the assumptions on the convergence of~$f_k$ and~\eqref{tfghdsjkcnxzm5647567438fghdj2}, we deduce that~$f_k+h_k\rightharpoonup  f$ in~$W^{1,q}(\Omega)$.
	
	Besides, by~\cite[Corollary~3.10]{MR3900821}, we have that 
	$$\sup_k \norm{\tilde{u}_k}_{L^\infty(\R^n)} \leq \max\left\{
	\left(1+c_W\sup_k \epsilon_k^{2s}\norm{f_k}_{L^\infty(\Omega)}\right)^{1/(p-1)}, \sup_k \norm{g_k}_{L^\infty(\R^n)} \right\}<+\infty,$$
	for some~$c_W>0$ depending only on~$W$.
	In particular, 
	$$\sup_k \norm{u_k}_{L^\infty(\Omega)}\leq \sup_k \norm{\tilde{u}_k}_{L^\infty(\Omega)} <+\infty.$$
	Thus, we infer that
	\begin{equation} \label{eq::seminorm_estimate}
		\begin{split}
			\K(\tilde{u}_k,\Omega) &= \frac{1}{2}\int_\Omega\int_\Omega \frac{|\tilde{u}_k(x)-\tilde{u}_k(y)|^2}{|x-y|^{n+2s}}dxdy + \int_\Omega\int_{\Omega^c} \frac{|\tilde{u}_k(x)-\tilde{u}_k(y)|^2}{|x-y|^{n+2s}}dxdy\\
			&= \frac{1}{2}\int_\Omega\int_\Omega \frac{|u_k(x)-u_k(y)|^2}{|x-y|^{n+2s}}dxdy + \int_\Omega\int_{\Omega^c} \frac{|u_k(x)-\tilde{u}_k(y)|^2}{|x-y|^{n+2s}}dxdy\\
			&=\K(u_k,\Omega)+ \int_\Omega\int_{\Omega_\eta^c} \frac{|u_k(x)-\tilde{u}_k(y)|^2-|u_k(x)-u_k(y)|^2}{|x-y|^{n+2s}}dxdy\\
			&\leq \K(u_k,\Omega)+  \int_\Omega\int_{\Omega_\eta^c} \frac{|\tilde{g}_k(y)-g_k(y)|\,|\tilde{g}_k(y)+g_k(y)-2u_k(x)|}{|x-y|^{n+2s}}dxdy\\
			&\leq \K(u_k,\Omega)+C,
		\end{split}
	\end{equation}
	for some positive~$C=C(\Omega,\eta)$. 
	
	As a result, the assumption~\eqref{mnbvcxz21345678jhgfdsurieowjdh}, together with~\eqref{eq::seminorm_estimate}, entails that~
	$$\sup_k\left(\F_{\epsilon_k}(\tilde{u}_k,\Omega)-\int_\Omega (f_k+h_k)\tilde{u}_kdx\right) <+\infty.$$

	Thus, we have that~$\tilde{u}_k$ satisfies the assumptions\footnote{We point out that, the assumption~$|g|=1$ almost everywhere in~$\Omega^c$ in~\cite[Theorem~5.1]{MR3900821} is in general unnecessary and it is used there only to obtain~$u=\chi_E-\chi_{E^c}$ in~$\R^n$, for some set~$E$, and hence~$\K(u,\Omega)=\Per_{s}(E,\Omega)$.} of~\cite[Theorem~5.1]{MR3900821}, and from the application of this result we deduce the existence of a (not relabeled) subsequence and a function~$u$ such that~$\tilde{u}_k\to u$ in~$L^1_{\loc}(\R^n)$, $u=\chi_E-\chi_{E^c}$ in~$\Omega$, for some set~$E\subseteq\R^n$, and
	\begin{equation*}
		\delta\K(u,\Omega)[X] = \kappa\int_{E\cap\Omega} \mbox{div}(fX)dx \quad\text{for every }X\in\cont^1_c(\Omega,\R^n), 
	\end{equation*}
	for some~$\kappa>0$ depending only on~$n$ and~$s$.
	
	Furthermore, since~$\tilde{u}_k=u_k$ in~$\Omega$ and~$g_k$, $\tilde{g}_k\to g$ in~$L^1_{\loc}(\Omega^c)$, we also have that~$u_k\to u$ in~$L^1_{\loc}(\R^n)$.
\end{proof}


\section[Proof of Theorem~\ref{th::conv_crit_points}\\ {\footnotesize Hybrid mean curvature of the phase interface under a mass constraint}]{Proof of Theorem~\ref{th::conv_crit_points} - Hybrid mean curvature of the phase interface under a mass constraint}
\label{sec::conv_crit_points}

Here, we prove that a non-local solution that presents a phase interface under a mass constraint has constant hybrid mean curvature, as claimed in Theorem~\ref{th::conv_crit_points}.  

\begin{proof}[Proof of Theorem~\ref{th::conv_crit_points}]
	Let~$K\comp\Omega$ and let~$\eta>0$ be such that
	$$ 	K_\eta := \{x\in\R^n \text{ s.t. } {\rm{dist}}(x,K)<\eta\}
	\subset\Omega.$$ 
	We point out that if~$u_\epsilon$ is a critical point of~$\G_\epsilon$ in~$\Omega$, then~$u_\epsilon$ is a weak solution of
	\begin{equation} \label{eq::weak_eq_u_epsilon}
		\begin{cases}
			(-\Delta)^sv+\epsilon^{-2s}W'(v)=-\mu_\epsilon \quad&\text{in }\Omega,\\
			v=u_\epsilon \quad&\text{in }\Omega^c.			
		\end{cases}
	\end{equation}
	Thus, by~\cite[Propositions~2.8 and~2.9]{silvestre_regularity},
	we have that~$u_\epsilon$ is smooth in~$\Omega$. 
	
	In particular, $u_\epsilon$ weakly solves
	\begin{equation} \label{eq::weak_eq_u_epsilon2}
		\begin{cases}
			(-\Delta)^sv+\epsilon^{-2s}W'(v)=-\mu_\epsilon \quad&\text{in }K,\\
			v=u_\epsilon \quad&\text{in }K^c.			
		\end{cases}
	\end{equation}
	Therefore, defining~$g_\epsilon:=u_\epsilon$ in~$K^c$, we have that~$g_\epsilon\in\cont^{0,1}(K_\eta)$ (recall that~$K_\eta\subset\Omega$) and that~$\sup_\epsilon\norm{g_\epsilon}_{L^\infty(\R^n)}\leq M$. Moreover, by Theorem~\ref{th::energy_continuity}, we get that~$g_\epsilon\to g$ in~$L^1_{\loc}(K^c)$ (up to a subsequence), for some function~$g$. 
	
	Also, setting~$f_\epsilon:=-\mu_\epsilon$, we obtain that~$f_\epsilon\in\cont^\infty(K)$ and it satisfies
	$$ \sup_k \left(\epsilon^{2s}\norm{f_\epsilon}_{L^\infty(K)}+\norm{f_\epsilon}_{W^{1,n/2}(K)}\right)\leq C_{n,\Omega}\sup_k\norm{f_\epsilon}_{L^\infty(\R^n)} <+\infty.$$
	Hence, thanks to Theorem~\ref{th::millot_sire_wang} and the uniqueness of the~$L^1_{\loc}$-limit, we conclude that there exists a function~$u$ such that~$u_\epsilon\to u$ in~$L^1_{\loc}(\R^n)$ (up to a subsequence), and~$u=\chi_E-\chi_{E^c}$ in~$\Omega$, for some set~$E$. Moreover,
	\begin{equation*}
		\delta\K(u,\Omega)[X] = \kappa\int_{E\cap\Omega} \mbox{div}(\mu X)dx \quad\text{for every }X\in\cont^1_c(\Omega,\R^n), 
	\end{equation*}
	for some~$\kappa>0$ depending only on~$n$ and~$s$.
\end{proof}

\appendixtitleon
\appendixtitletocon
\begin{appendices}
	
	\section{Non-existence of solutions of an over-determined Allen-Cahn equation} \label{append::over_determined}
	
	We show that a prescribed-mass Allen-Cahn equation with Dirichlet boundary conditions is an over-determined problem, where the existence of solutions is not guaranteed, even in dimension one. To do this, we exhibit an example of such an equation with no solutions, both in the classical and in the fractional case.
	
	\begin{ex}[A counter-example in the classical case]
		
		Let us consider the ordinary differential equation
		\begin{equation} \label{eq::broken_ode}
			\begin{cases}
				&u'' = u^3-u\quad\text{in }(-1,1),\\
				&u(-1)=u(1)=-1.
			\end{cases}
		\end{equation}
		We look for solutions of~\eqref{eq::broken_ode} satisfying the mass constraint
		\begin{equation}\label{eq::broken_mass_constraint}
			\int_{-1}^1 u\,dx = 0.
		\end{equation}
		Observe that a solution~$u$ of~\eqref{eq::broken_ode} is negative at~$\{-1,1\}$, but has average~$0$. Thus, we infer that~$u$ has a
		strictly positive maximum in~$(-1,1)$, i.e. there exists~$\overline{x}\in(-1,1)$ such that
		\begin{equation}\label{fesdrfyhtfg87654grhwdkbdsnmfdbsn}
			\mu := u(\overline{x}) = \max_{[-1,1]} u > 0.
		\end{equation} 
		Since~$\overline{x}$ is a maximum point for~$u$, we have that 
		\begin{equation*}
			0 \geq u''(\overline{x}) = u^3(\overline{x})-u(\overline{x}) = \mu^3-\mu=\mu(\mu^2-1),
		\end{equation*}
		which entails that~$\mu\in(0,1]$. 
		
		Now we observe that, for all~$x\in(-1,1)$,
		\begin{eqnarray*}
			\frac{d}{dx}\left(\frac12|u'(x)|^2-\frac14u^4(x)+\frac12u^2(x)\right)
			=u'(x)\Big( u''(x)-u^3(x)+u(x)\Big)=0.
		\end{eqnarray*}
		As a consequence,
		\begin{eqnarray*}&&-\frac{\mu^4}4+\frac{\mu^2}2=
			\frac12|u'(\overline x)|^2-\frac14u^4(\overline x)+\frac12u^2(\overline x)
			\\&&\qquad=\lim_{x\to1}\left(\frac12|u'(x)|^2-\frac14u^4(x)+\frac12u^2(x)\right)
			\ge \lim_{x\to1}\left( -\frac14u^4(x)+\frac12u^2(x)\right)=\frac14,
		\end{eqnarray*} namely
		$$ \mu^2(2-\mu^2)\ge1.$$
		This entails that~$\mu=1$ and
		$$ \lim_{x\to1}|u'(x)|^2=0,$$
		and therefore~$u'(1)=0$. 
		
		As a result, by the uniqueness of the Cauchy problem, we obtain that~$u$
		must be identically equal to~$-1$, in contradiction with~\eqref{fesdrfyhtfg87654grhwdkbdsnmfdbsn}.
	\end{ex}
	
	\begin{ex}[A counter-example in the non-local case]
		Let~$s\in(0,1)$ and let~$u_0$ be the unique non-trivial global minimizer of the functional~$J_1(\cdot,\Omega)$ (defined in~\eqref{eq::1D_functional}). According to~\cite[Theorem~2]{MR3081641}, we know that~$u_0:\R\to[-1,1]$ is a strictly increasing odd function.
		
		We consider the pseudo-differential problem
		\begin{equation} \label{eq::s_ode}
			\begin{cases}
				&(-\Delta)^su = u-u^3 \quad\text{in }(-1,1),\\
				&u=u_0\quad\text{in }\R\setminus (-1,1),
			\end{cases}
		\end{equation}
		and we claim that 
		\begin{equation} \label{clam::unique_sol}
			\mbox{$u_0$ is the unique solution of~\eqref{eq::s_ode}.}
		\end{equation}
			
			By symmetry, we also have that
			$$ \int_{-1}^{1}u_0\,dx=0.$$
			Therefore, we deduce that the over-determined problem~\eqref{eq::s_ode} subject to the mass constraint 
			\begin{equation}
				\int_{-1}^1 u\,dx = \eta,
			\end{equation}
			has no solutions, for any fixed~$\eta>0$. 
			
			To show~\eqref{clam::unique_sol}, we suppose by contradiction that 
			\begin{equation}\label{vbcnxieo3yr4378r23456789}
				{\mbox{there exists a continuous
						function~$v\not\equiv u_0$ which solves~\eqref{eq::s_ode}.}}
			\end{equation} 
			Then, by sliding horizontally $u_0$ and employing the maximum principle for pseudo-differential operators, we infer that $v$ must coincide with $u_0$, leading to a contradiction.
			
		With this strategy in mind, we first observe that 
		\begin{equation}\label{clam::unique_solBIS}
			{\mbox{$v< 1$ in $[-1,1]$.}}	
		\end{equation}
		Indeed, let~$\overline{x}$ be such that
		$$ v(\overline{x}) = \max_{[-1,1]}v .$$
		In particular, since~$u_0$ is increasing, we have that~$\overline{x}\neq -1$. Also, $v(\overline{x})\geq u_0(1)>0$. 
		
		Now, if~$\overline{x}=1$, then~$v(\overline{x})=u_0(\overline{x})< 1$, which implies~\eqref{clam::unique_solBIS}. 
		
		If instead~$\overline{x}\in(-1,1)$, then
		$$ 0 \leq \int_{\R} \frac{v(\overline{x})-v(y)}{|\overline{x}-y|^{1+2s}}dy = v(\overline{x})-v^3(\overline{x})
		=v(\overline{x})\big(1-v^2(\overline{x})\big),$$
		which reads~$v(\overline{x})\leq 1$. However, if~$v(\overline{x})=1$, then we would have
		$$ 0\le\int_{\R} \frac{1-v(y)}{|\overline{x}-y|^{1+2s}}dy = 0,$$
		which entails that~$v\equiv1$, against the fact that~$v=u_0$ in~$\R\setminus (-1,1)$. Hence, we conclude that~$v(\overline{x})<1$
		and~\eqref{clam::unique_solBIS} is established.
		
		Now, the idea is to slide~$u_0$ to the left until it lies completely above~$v$, and then slide back~$u_0$ until it touches~$v$ (from above) for the first time. To do this, we consider horizontal translations of~$u_0$, defined as~$(u_0)_t(x):=u_0(x+t)$, with~$t\in\R$. In light of~\eqref{clam::unique_solBIS}, there exists~$T>0$ such that~$(u_0)_T>v$. Then, we define
		\begin{equation}\label{qwe0wdufjovl}
			\tau := \inf\big\{\theta \mbox{ s.t. } u_0(x+\theta)>v(x),\mbox{ for all }x\in\R\big\}\geq0.
		\end{equation}
		We observe that, by construction,
		\begin{equation}\label{vbncmxeuwiruwrty75993}
			{\mbox{$(u_0)_{\tau}(x)\ge v(x)$ for all~$x\in\R$.}}
		\end{equation}
		
		We claim that
		\begin{equation}\label{wofejvbn9-234500}\tau=0.\end{equation}
		For this, we argue by contradiction and we suppose that
		\begin{equation}\label{wofejvbn9-2345} \tau>0.\end{equation}
		
		We point out that
		\begin{equation}\label{qwe0wdufjovl2}
			{\mbox{there exists~$x_\tau\in[-1,1]$ such that~$(u_0)_{\tau}(x_\tau)=v(x_\tau)$.}}
		\end{equation}
		Indeed, if not, by~\eqref{vbncmxeuwiruwrty75993} we have that
		$$\min_{[-1,1]}\big((u_0)_{\tau}-v\big)>0.$$
		Hence, by continuity and in view of~\eqref{wofejvbn9-2345}, there exists~$\tau_\star\in(0,\tau)$ such that
		\begin{equation} \label{vbncmxeuwiruwrty759932}
			\mbox{$(u_0)_{\tau_\star}-v>0$ in~$[-1,1]$.}
		\end{equation}
		
		Then, by the monotonicity of~$u_0$, we have that, for all~$x\in\R\setminus(-1,1)$,
		\begin{equation*} (u_0)_{\tau_\star}(x)=u_0(x+\tau_\star)>u_0(x)=v(x). \end{equation*}
		Combining this with~\eqref{vbncmxeuwiruwrty759932}, we see that~$(u_0)_{\tau_\star}-v>0$ in the whole of~$\R$. This violates
		the definition in~\eqref{qwe0wdufjovl} and proves~\eqref{qwe0wdufjovl2}.
		
		Besides, notice that, since~$u_0$ satisfies 
		$$ (-\Delta)^s u_0 = u_0-u_0^3\quad\mbox{in }\R,$$
		then for every~$t\in\R$ and for every~$x\in\R$ we have that
		\begin{equation*}
			\begin{split}
				&	(-\Delta)^s (u_0)_t (x) = \int_{\R} \frac{(u_0)_t(x)-(u_0)_t(y)}{|x-y|^{1+2s}}\,dy = \int_{\R} \frac{u_0(x+t)-u_0(y)}{|(x+t)-y|^{1+2s}}\,dy \\
				&\qquad\qquad= u_0(x+t)-u_0^3(x+t).
			\end{split}
		\end{equation*}
		
		Thus, recalling~\eqref{qwe0wdufjovl2},
		\begin{equation*}
			\begin{split}
				&\int_{\R} \frac{(u_0)_\tau(x_\tau)-v(y)}{|x_\tau-y|^{1+2s}}\,dy = \int_{\R} \frac{v(x_\tau)-v(y)}{|x_\tau-y|^{1+2s}}\,dy=v(x_\tau)-v^3(x_\tau)\\
				&\qquad\qquad = (u_0)_\tau(x_\tau)-(u_0)_\tau^3(x_\tau)
				=\int_{\R} \frac{(u_0)_\tau(x_\tau)-(u_0)_\tau(y)}{|x_\tau-y|^{1+2s}}\,dy.
			\end{split}
		\end{equation*}
		{F}rom this and~\eqref{vbncmxeuwiruwrty75993}, we obtain that				 
		\begin{equation*}
			0\leq \int_{\R} \frac{(u_0)_\tau(y)-v(y)}{|x_\tau-y|^{1+2s}}\,dy = 0,
		\end{equation*}
		which implies that~$v=(u_0)_\tau$ almost everywhere. However, this is against the fact that~$v=u_0$ in~$\R\setminus (-1,1)$.
		This contradiction proves~\eqref{wofejvbn9-234500}.
		
		{F}rom~\eqref{vbncmxeuwiruwrty75993} and~\eqref{wofejvbn9-234500}, we thereby find that~$u_0\geq v$ in~$\R$.
		
		In a similar manner, one can show that~$u_0\leq v$ in~$\R$, and hence~$v=u_0$ in~$\R$.
		This is a contradiction with~\eqref{vbcnxieo3yr4378r23456789}, and
		therefore~\eqref{clam::unique_sol} holds true, as claimed.		
		
	\end{ex}
	
	
	\section{Extension of $\Gamma$-convergence results to $X_M$} \label{append::X_M}
		
		For the convenience of the reader, in this section, we provide a self-contained justification for extending the results of~\cite{savin_valdinoci_gamma_conv} to the space $X_M$, as claimed in Remark~\ref{rem::X_M}.
		
		We remark that the results presented in~\cite{savin_valdinoci_gamma_conv} are stated and proved for functionals~$\F_\epsilon$ and~$\F$ defined on the space 
		$$X := \{u\in L^1_{\loc}(\R^n) \mbox{ s.t. }\norm{u}_{L^\infty(\R^n)}\leq 1\}.$$
		Recalling~\eqref{eq::def_X_M}, we observe that $X\subseteq X_M$, for every $M\geq1$.
		
		The two main results we are concerned about are the following:
		\begin{theorem} \label{th::gamma_conv_SV}
			Let~$s\in(0,1)$. Then, the sequence of functionals~$\F_\epsilon$ $\Gamma$-converges to~$\F$, i.e. for any~$u\in X_M$ we have
			\begin{enumerate}[(i)]
				\item for any sequence~$\{u_\epsilon\}_\epsilon\subset X_M$ converging to~$u$ in~$L^1_{\loc}(\R^n)$,
				$$ \F(u,\Omega) \leq \liminf_{\epsilon\to0^+} \F_\epsilon(u_\epsilon,\Omega) ,$$
				\item there exists a sequence~$\{u_\epsilon\}_\epsilon\subset X_M$ converging to~$u$ in~$L^1_{\loc}(\R^n)$ such that 
				$$ \F(u,\Omega) \geq \limsup_{\epsilon\to0^+} \F_\epsilon(u_\epsilon,\Omega) .$$
			\end{enumerate}
		\end{theorem}
		
		\begin{theorem} \label{th::min_conv_SV}
			If $\F_\epsilon(u_\epsilon,\Omega)$ is uniformly bounded for a sequence of $\epsilon\to0^+$, then there exist a convergent subsequence of $\{u_\epsilon\}_\epsilon$ and a set $E$ such that
			\begin{equation} \label{eq::compactness_SV}
				u_\epsilon\to u_\star:=\chi_E-\chi_{E^c}\quad\mbox{in }L^1(\Omega).
			\end{equation}
			
			Moreover, if~$u_\epsilon$ minimizes~$\F_\epsilon$ in~$\Omega$, then
			\begin{enumerate}[(i)]
				\item if~$s\in(0,1/2)$ and~$u_\epsilon$ is weak* convergent to some~$u_0$ in~$L^\infty(\Omega^c)$, then~$u_\star$ minimizes~$\F$ among all the functions that coincide with~$u_0$ in~$\Omega^c$;
				\item if~$s\in[1/2,1)$, then~$u$ minimizes~$\E$. Also, for any~$K\comp\Omega$, it holds
				\begin{equation} \label{eq::local_conv_SV}
					\limsup_{\epsilon\to0^+} \F_\epsilon(u_\epsilon,K) \leq c_\star\Per(E,\overline{K}) ,
				\end{equation}
				where~$c_\star=c_\star(n,s,W)$ is as in~\eqref{eq::F}.
			\end{enumerate}
		\end{theorem}
		
		We observe that
		Theorems~\ref{th::gamma_conv_SV}
		and~\ref{th::min_conv_SV}
		here correspond, respectively, to
		Theorems~1.2 and~1.3 in~\cite{savin_valdinoci_gamma_conv}
		when~$M=1$
		and we now briefly discuss the modifications needed in our setting.
		
		First, observe that if $\{u_\epsilon\}_\epsilon$ is a sequence in $X_M$, we still have a uniform upper bound for $\norm{u_\epsilon}_{L^\infty(\Omega)}$. Indeed, recall that $v\in X_M$ yields $|v|\leq M$ almost everywhere in $\Omega$, and not necessarily in the whole of $\R^n$.
		
		Then, by inspection of the proofs in~\cite{savin_valdinoci_gamma_conv}, we notice that the argument to show Theorem~\ref{th::gamma_conv_SV}-(i) remains valid without requiring $\{u_\epsilon\}_\epsilon$ to be uniformly bounded in~$\Omega^c$ and extends to any space $X_M$, with $M\geq1$, with no modifications. Additionally, since $X\subseteq X_M$, Theorem~\ref{th::gamma_conv_SV}-(ii) also follows immediately for every $X_M$ by choosing the same recovery sequences.
		
		On the other hand, the proof of Theorem~\ref{th::min_conv_SV} for $X_M$ is slightly more delicate and requires minor adjustments. In particular, a global uniform bound for $u_\epsilon$ is needed to prove Theorem~\ref{th::min_conv_SV}. This is ensured by Lemma~\ref{lemma::unif_bound_min}, which establishes that $\norm{u_\epsilon}_{L^\infty(\R^n)}\leq M$ whenever $u_\epsilon$ is a minimizer for $\F_\epsilon$ in~$X_M$, as assumed in Theorem~\ref{th::min_conv_SV}.
		The other main tool to show Theorem~\ref{th::min_conv_SV} is the following auxiliary result, which yields \eqref{eq::compactness_SV}.
		
		\begin{prop}\label{prop::bound_conv}
			Let $\Omega$ be an open, bounded subset of $\R^n$ and $u_\epsilon\in X_M$, with $\epsilon>0$. If
			$$ \liminf_{\epsilon\to0^+} \F_\epsilon(u_\epsilon,\Omega)\leq C_0<+\infty,$$
			for some constant $C_0$, then there exist a subsequence of $\{u_\epsilon\}_\epsilon$ and a set $E\subseteq \R^n$ such that 
			$$u_\epsilon\to\chi_E-\chi_{E^c}\quad\mbox{in }L^1(\Omega)$$		
			and
			$$ \Per(E,\Omega)<+\infty.$$
		\end{prop} 
		
		This result reduces to~{\cite[Proposition 3.3]{savin_valdinoci_gamma_conv}} when~$M=1$ and, once established,
		it would lead to Theorem~\ref{th::min_conv_SV}.
		
		Thus, to obtain Proposition~\ref{prop::bound_conv} in the generality needed here,
		we point out that, after replacing the estimate on display between Equations (3.8) and (3.9) in \cite{savin_valdinoci_gamma_conv} with
		\begin{equation*}
			\left|\{\left|1-|u_\epsilon|\right|\geq\sigma\}\right| \leq C(\sigma)\int_{\Omega} W(u_\epsilon)\,dx \leq C(\sigma,C_0)\epsilon^{1/2}\leq\frac{\delta}{8} ,
		\end{equation*}
		the argument proceeds without further changes.
		
\end{appendices}		

\section*{Acknowledgments} \label{sec::acknowledgments}	
\addcontentsline{toc}{section}{\nameref{sec::acknowledgments}}
Serena Dipierro and Enrico Valdinoci are members of the Australian Mathematical Society (AustMS).
Serena Dipierro and Riccardo Villa are supported by the Australian Research Council Future Fellowship FT230100333 ``New perspectives on nonlocal equations''.
Enrico Valdinoci is supported by the Australian Laureate Fellowship FL190100081 ``Minimal surfaces, free boundaries and partial differential equations''.
Riccardo Villa is supported by a Scholarship for International Research Fees at the University of Western Australia.

\addcontentsline{toc}{section}{References}
\nocite{*}
\printbibliography

\end{document}